 \documentclass[final,3p,times]{elsarticle}

\usepackage{amssymb}

\journal{Stochastic Processes and their Applications}

\usepackage{amsmath, amsthm, amssymb} 
\usepackage{bm} 

\usepackage{epsfig} 
\usepackage{graphicx} 

\def\bea{\begin{eqnarray}}
\def\eea{\end{eqnarray}}
\def\bean{\begin{eqnarray*}}
\def\eean{\end{eqnarray*}}
\renewcommand\eqref[1]{(\ref{#1})}

\def\R{\mathbb{R}}
\def\N{\mathbb{N}}
\def\Z{\mathbb{Z}}

\def\Ind#1{\,\mathbb{I}\{#1\}\,}

\theoremstyle{plain}
\newtheorem{theorem}{Theorem}
\newtheorem{proposition}[theorem]{Proposition}
\newtheorem{corollary}[theorem]{Corollary}
\newtheorem{lemma}[theorem]{Lemma}
\theoremstyle{definition}
\newtheorem{definition}[theorem]{Definition}
\def\CP{{N}} 

\def\dt{{h}} 

\def\dtt{{\bar{h}}} 

\def\inNat{{\in \N}}
\def\inNatII{{\ge 1}}

\def\MultI{{\mathcal{I}}}
\def\MultII{{\mathcal{J}}}
\def\MultSet{{\zeta}}
\def\MultCar{{C}}
\def\MultInd{{c}}

\def\BlockSet{{\mathcal{B}}}
\def\BlockInd{{b}}

\def\c{{c}} 

\def\SIP{{\Gamma}}

\def\Tf{U}
\def\tf{u}

\def\Ts{V}
\def\ts{v}

\ifx \q \undefined
  \def \q[#1][#2][#3]{q_{#2}^{#3}(#1)}
  \def \qt[#1][#2][#3]{\tilde q_{#2}^{#3}(#1)}
\fi 

\def \papertitle{Compound Markov counting processes and their applications to modeling infinitesimally over-dispersed systems}

\def \paperkeywords{continuous time; counting Markov process; birth-death process; environmental stochasticity; infinitesimal over-dispersion; simultaneous events}

\def \paperabstract{
We propose an infinitesimal dispersion index for Markov counting processes.
We show that, under standard moment existence conditions,
a process is infinitesimally (over-) equi-dispersed if, and only if, it is simple (compound), i.e. it increases in jumps of one (or more) unit(s), even though infinitesimally equi-dispersed processes might be under-, equi- or over-dispersed using previously studied indices.
Compound processes arise, for example, when introducing continuous-time white noise to the rates of simple processes resulting in L\'{e}vy-driven SDEs.
We construct multivariate infinitesimally over-dispersed compartment models and queuing networks, suitable for applications where moment constraints inherent to simple processes do not hold.
}

\begin{document}

\begin{frontmatter}



\title{\papertitle}

\author{C. Bret\'{o}\fnref{label1}}
\fntext[label1]{corresponding author; Tel. +34916245855; Fax:+34916249849}
\ead{cbreto@est-econ.uc3m.es}
\address{Department of Statistics, Universidad Carlos III of Madrid, C/ Madrid 126, 28903, Madrid, Spain}

\author{E. L. Ionides}
\ead{ionides@umich.edu}
\address{Department of Statistics, University of Michigan, 1085 South University Avenue, Ann Arbor, MI 48109-1107}


\begin{abstract}
\paperabstract
\end{abstract}

\begin{keyword}
\paperkeywords

\end{keyword}

\end{frontmatter}



\section{Introduction}\label{sec:intro}
Continuous-time stochastic processes are widely used as a modeling tool for studying dynamical systems in different fields.
Most continuous-time processes proposed in the literature belong to one of two large families: real-valued processes which can be written as solutions to stochastic differential equations \citep{karlin1975,oksendal1998} and discrete-valued processes defined via counting processes \cite{daley2003, snyder1991, cox1980} or Markov chains \cite{bremaud1999}.
In this paper, we focus on the intersection between counting processes and Markov processes, namely Markov counting processes (MCPs from this point onward).
MCPs are building blocks for models which are heavily used in biology (in the context of compartment models) and engineering (in the context of queues and queuing networks) as well as in many other fields.

A \emph{counting process} is a continuous-time, non-decreasing, non-negative, integer-valued stochastic process.
The counting process is said to count \emph{events} each of which has an associated \emph{event time}.
A counting process is \emph{simple} if, with probability one, there is no time at which two or more events occur simultaneously.
A process which is not simple is called \emph{compound}.
Simpleness is a convenient, and therefore widely adopted, property for both the theory and applications of counting processes \citep{daley2003}.
The Markov property is also a convenient and widespread property of stochastic models.
However, we will show that simple MCPs, combining these two attractive properties, have severe limitations in terms of the range of possible relationships between their infinitesimal mean and variance.
Previous approaches to negotiate this difficulty have centered on sacrificing the Markov property rather than simpleness.
However, there are theoretical and practical attractions to the alternative strategy of maintaining the Markov property while allowing for simultaneous events.
Investigating such models is the topic of this paper.

The ratio of the variance to the mean of a random variable is called its \emph{dispersion}.
Many well-known integer-valued distributions have dispersion constraints.
These constraints are often not reproduced in data from applications, the data typically having additional variance and therefore being termed \emph{over-dispersed} \citep{mccullagh1989}.
The same issues arise in integer-valued stochastic processes \citep{brown1998} and, as a result, there is a considerable literature devoted to extending otherwise appealing models which are unable to reproduce observed variability.
Typically, over-dispersion has been studied via defining stochastic processes in which some parameters are themselves modeled as stochastic in order to  produce additional variability.
This idea has been widely applied since the pioneering work of Greenwood and Yule \cite{greenwood1920}, which derived the over-dispersed negative binomial distribution as a mixture of the Poisson distribution with a gamma-distributed parameter.
Another early contribution is the Cox process \citep{cox1955}, also known as doubly-stochastic Poisson  process \citep{cox1980,snyder1991,daley2003}.
Some recent work has considered stochastic parameters for continuous-time Markov chains \cite{economou2005} and for non-Markovian processes \cite{swishchuk2003}.
Marion and Renshaw \cite{marion2000} and Varughese and Fatti \cite{varughese2008} studied over-dispersion generated by standard birth-death processes with diffusion-driven rates, focusing on population dynamics applications.
Both \cite{marion2000} and \cite{varughese2008} proposed a mean-reverting Ornstein-Uhlenbeck process for the driving random environment.
Compound counting processes have been studied in the literature on batch processes \cite{ormeci2005}, but we are not aware of a previous investigation of infinitesimal dispersion in this context.
To our knowledge, the first general class of infinitesimally over-dispersed MCPs was proposed by Bret\'{o} et al. \cite{breto2009-aoas}.
They achieved over-dispersion by introducing white noise to rates of a multivariate process constructed via simple death processes, which was shown to result in the possibility of simultaneous events.
The main goal of this paper is to generalize the model of \cite{breto2009-aoas} by presenting a systematic investigation of over-dispersed models via compound MCPs.
In particular, those defined by L\'{e}vy-driven stochastic differential equations \cite{applebaum2004} resulting from introducing continuous-time white noise in the rate of simple MCPs via Kolmogorov's differential equations.
The applications of MCPs are too diverse to cover systematically here.
One concrete example, which has been a motivation for our work \cite{breto2009-aoas}, is the study of infectious disease dynamics.
Discrete-state Markov processes have proven useful models for studying many infectious disease transmission systems, and are central to current understanding of the spread of such diseases through populations \citep{keeling2009}.
However, standard disease models are constructed via simple MCPs and therefore struggle to match the statistical properties observed in data.
Recent advances in statistical inference methodology \citep{ionides2006-pnas,andrieu2010} have permitted fitting more general models, based on compound MCPs, to data \citep{breto2009-aoas, he2009}.
At least in this context, the substantial scientific consequences of adequately modeling over-dispersion in stochastic processes are consistent with the widely recognized importance of over-dispersion for drawing correct inferences from integer-valued regression models \citep{mccullagh1989}.

As concrete examples of models defined by L\'{e}vy-driven Kolmogorov's differential equations, we compute infinitesimal moments and infinitesimal probabilities for various specific novel models.
The availability of infinitesimal probabilities makes possible exact simulation, and exact methods are particularly appropriate when dealing with small counts, which arise naturally in some applications.
In infectious disease applications, for example, small counts arise at the start of an epidemic, which is a critical period for identifying and controlling the disease transmission.
Exact simulation of MCPs can be computationally demanding for processes with a very large number of events.
In this cases, it is standard to use approximations which are more affordable computationally but require some diagnostics to investigate the validity of the approximation.
To this end, both Euler-Maruyama time discretizations of MCPs and diffusion approximations have been proposed in the literature \cite{breto2009-aoas, he2009,ionides2006-pnas, king2008, marion2000, varughese2008, fan1999}.
Several algorithms have been proposed  in which two simulation methods are used, an exact one for small counts and a faster, approximate one for larger counts \cite{haseltine2002}.
In order to use combined algorithms of this type, it is necessary to choose a diffusion approximation, given some MCP.
Diffusions are defined in a straightforward way in terms of infinitesimal moments.
Requiring that the MCP and the proposed diffusion approximation have common infinitesimal moments gives a natural approach for such an approximation, giving further motivation for the study of infinitesimal moments of MCPs.

A second goal of this paper is to propose the use of an infinitesimal dispersion index for counting processes in conjunction with standard indices.
This provides a simple measure of dispersion, combining attractive theoretical properties with scientific interpretability, which is desirable when considering candidate processes for applications.
Markov processes specified as the solution to stochastic differential equations are naturally characterized by their infinitesimal mean and variance \citep{karlin1975}.
However, these infinitesimal moments have not been studied in the context of counting processes, perhaps because, as we will show, in the case of simple MCPs the infinitesimal variance is constrained to be equal to the infinitesimal mean.
Instead, interest has focused on dispersion properties of increments of counting processes over fixed time windows, which we call \emph{integrated dispersion} to distinguish it from infinitesimal dispersion.
The study of integrally over-dispersed counting processes has a long history, going back at least to the start of the twentieth century \citep{student1907} and continuing up to the present \citep[e.g.,][]{balakrishnan2008}.
Integrated dispersion has undoubtedly an interest of its own, in particular if the integration window is chosen according to some specific criterion (possibly motivated in applications by scientific evidence).
Because of this window dependence, integrated dispersion may give a distorted representation of a process, in the same way that discretizing a continuous-time process at different resolutions might give very different pictures.
In particular, we show that all of integrated over-, equi- and under-dispersion may occur for infinitesimally equi-dispersed processes.
By contrast, infinitesimal dispersion provides an intuitive and theoretically attractive measure which has already proven its worth in the study of real-valued Markov processes.

In Section~\ref{sec:notation} we investigate the infinitesimal moments of simple and compound MCPs, and compare them with previously studied measures of integrated dispersion.
Then, in Section~\ref{sec:OMCP} we propose several novel over-dispersed compound MCPs.
In Section~\ref{sec:notation-mult}, we define multivariate versions of the dispersion indices of Section~\ref{sec:notation} and find sufficient and necessary conditions for infinitesimal equi- and over-dispersion.
Finally, in Section~\ref{sec:OMCS} we show how the univariate MCPs of Section~\ref{sec:OMCP} may be used as building blocks for more complex processes, such as compartment models or queuing networks, which inherit the desired dispersion properties.
We conclude with Section~\ref{sec:disc}, where we discuss some conceptual and practical issues in modeling via compound MCPs.


\section{Dispersion of Markov counting processes}\label{sec:notation}

One can study dispersion in the context of non-Markovian processes, but several considerations have led us to focus on the Markov case here.
Firstly, there is less room for debate over the definition of appropriate measures of dispersion for Markov processes.
Secondly, the extensively studied theory of Markov chains \citep{bremaud1999} allows us to avoid explicitly discussing measure-theoretic issues while being guaranteed that there are no difficulties concerning the existence and construction of the processes in question.
Thirdly, our later goal of studying over-dispersed Markov counting processes clearly does not necessitate a complete investigation of non-Markovian possibilities.
We comment on some non-Markovian situations in Section~\ref{sec:nonMarkov}.

Let $\{N(t):t \in \R^+\}$ be a time homogeneous Markov counting process, which we will refer to as $\{N(t)\}$.
By analogy with the terminology of infinitesimal and integrated moments of Section~\ref{sec:intro}, we define infinitesimal increment probabilities (or just \emph{infinitesimal probabilities}) of an MCP to be
\begin{eqnarray}
\label{eqn:CP-inf-prob}
\q[n,k][][] &\equiv& \lim\limits_{\dt \downarrow 0} \frac{P(\Delta N(t)=k|N(t)=n)}{\dt}.
\end{eqnarray}
These are also commonly referred to as the local characteristics of the transition semigroup, or the infinitesimal generator of the corresponding MCP \cite{bremaud1999}.
To clarify our notation, note that~\eqref{eqn:CP-inf-prob} are not actual probabilities but rather the appropriate limit of the \emph{integrated probabilities} $P(\Delta N(t)=k|N(t)=n)$.
Here $t, \dt \in \R^+$ and $k, n\;\inNat$ with $k\;\inNatII$.
The operator $\Delta$ acting on a stochastic process is defined as $\Delta N(t) = N(t+\dt) - N(t)$ and the dependence of $\Delta N(t)$ on $\dt$ is suppressed.
Following standard terminology for counting processes, we define the \emph{intensity} or \emph{(infinitesimal) rate} function of such a process to be
\[
\lambda(n) \equiv \lim\limits_{\dt \downarrow 0}\frac{1-P(\Delta N(t)=0|N(t)=n)}{\dt}.
\]
Note that in definition \ref{eqn:CP-inf-prob} we have allowed for simultaneous events, i.e. $\{N(t)\}$ need not be simple.
Simple processes may be fully specified via their rate function, which in this case is $\lambda(n)=\q[n,1][][]$, and this is also a measure of the intensity at which events occur.
A counting process \emph{jumps} whenever there is an event, and we call the times at which there is one or more event \emph{jump times}.
The jumps are of size one if the process is simple and might be of greater size if the process is compound.
We emphasize the difference between jump times and event times because these two concepts overlap in the specific case of simple counting processes but are in general distinct.
To specify a compound processes one needs to provide all the infinitesimal probabilities, since the rate $\lambda(n)$ corresponds only to the rate of jumps and is uninformative about the distribution of jump sizes. For such processes, the infinitesimal mean may be a superior measure of the intensity at which events occur.

We restrict ourselves to \emph{stable} and \emph{conservative} processes for which $\lambda(n)=\sum_{k\ge 1}q(n,k) <\infty$ for all $n$.
Markov processes satisfying these conditions form a very general class, and the MCP is then characterized by its infinitesimal probabilities \cite{bremaud1999}.
We also restrict ourselves to time homogeneous processes to add clarity to the concepts, results and proofs.
However, these can be readily generalized to the non-homogeneous case, for which the infinitesimal probabilities also depend on time.

Measures of dispersion which have previously been considered for counting processes include the variance to mean ratio $V[N(t)]/E[N(t)]$ (for example in \citep{gillespie1984}) and the difference $V[N(t)] - E[N(t)]$ (in \citep{brown1998}).
We will define the integrated dispersion index of $\{N(t)\}$ as
\begin{eqnarray}
\label{eqn:IntDisp}
 D_N(n_0,t) & \equiv & \frac{V[N(t)-N(0)|N(0)=n_0]}{E[N(t)-N(0)|N(0)=n_0]}.
\end{eqnarray}
Usually $n_0$ is assumed to be $0$ in which case $D_N$ corresponds to the standard dispersion index defined as a ratio.
Note however that \eqref{eqn:CP-inf-prob} defines $\{N(t)\}$ in infinitesimal terms.
This suggests the infinitesimal dispersion index which we define as
\begin{eqnarray}
\label{eqn:DispIndex} D_{dN}(n) & \equiv& \frac{\lim_{\dt \downarrow 0}\dt^{-1}V[N(t+\dt)-N(t)|N(t)=n]}{\lim_{\dt \downarrow 0} \dt^{-1}E[N(t+\dt)-N(t)|N(t)=n]} \equiv \frac{\sigma^2_{dN}}{\mu_{dN}},
\end{eqnarray}
as an alternative to $D_N$.
The numerator and denominator of~\eqref{eqn:DispIndex} are the standard definitions of \emph{infinitesimal variance} and \emph{infinitesimal mean} respectively \cite{karlin1975}.
Note that these two moments are conditional and that dependence of the infinitesimal moments on $n$ is suppressed.
By the algebraic properties of limits, $D_{dN}(n_0) = \lim_{t \downarrow 0}D_N(n_0,t)$ as long as the limit of the denominator of \eqref{eqn:DispIndex} exists.
A process has traditionally been considered over-dispersed when $V[N(t)] > E[N(t)]$. Analogously we define a process as infinitesimally (integrally) over-dispersed if $D_{dN}> 1(D_N > 1)$, for all $t$ and $n$, and define under- and equi-dispersion accordingly. For some processes, these conditions might not hold for all $t$ or all $n$. In this case we specify the subsets for which they hold.
Since we focus on infinitesimal properties, we will drop in the rest of the paper the term infinitesimal in order to simplify notation.
If we do not specify whether we refer to infinitesimal or integrated moments or dispersion, it should be understood that we mean the former.

In light of definition \ref{eqn:DispIndex}, it is interesting to find necessary and sufficient conditions characterizing dispersion of processes before considering construction of over-dispersed ones, which we proceed to do in Sections~\ref{sec:OMCP} and \ref{sec:OMCS}.
In this section, we establish for univariate MCPs sufficient conditions for equi-dispersion in Theorem~\ref{thm:suf-cond-equidisp} and necessary conditions for over-dispersion in Corollary~\ref{cor:nec-cond-overdisp}.
This provides a starting point for our investigation.
For example, it is immediate that compound Poisson processes (i.e., the class of compound MCPs with stationary independent increments) will be over-dispersed unless the event size distribution is degenerate at 1, in which case one is back to the simple Poisson processes.
We defer the more complete result of necessary and sufficient conditions for both equi- and over-dispersion to Section~\ref{sec:notation-mult} where we consider multivariate processes.

To understand what lies at the heart of equi-dispersion and of Theorem~\ref{thm:suf-cond-equidisp} consider the following expressions for the moments of the increments of a process $\{N(t)\}$:
\begin{eqnarray}
\nonumber E\big[\Delta N^r(t)|N(t)\big] &=& 0^{r}P\big(\Delta N(t){=}0|N(t)\big) + 1^rP\big(\Delta N(t){=}1|N(t)\big) + \sum\limits_{k=2}^{\infty}k^rP\big(\Delta N(t){=}k|N(t)\big)\\
\label{eqn:equi-disp-intuition}\lim_{\dt \downarrow 0}\frac{E\big[\Delta N^r(t)|N(t)\big]}{\dt} &=& \lim_{\dt \downarrow 0} \frac{P\big(\Delta N(t){=}1|N(t)\big)}{\dt} + \lim_{\dt \downarrow 0}\frac{\sum\limits k^rP\big(\Delta N(t){=}k|N(t)\big)}{\dt}.
\end{eqnarray}
It is straightforward that the difference between any two infinitesimal or integrated moments comes from terms in the sum corresponding to increments of size larger than one, i.e. to simultaneous events.

An immediate way to proceed to obtain sufficient conditions for equi-dispersion would be to require orderliness in the sense of Daley and Vere-Jones \citep[page 47]{daley2003}, i.e. $P\big(\Delta N(t) \geq 2\big) = o(\dt)$, and to investigate under which conditions the $\dt$ limit can be exchanged with the limit of the infinite sum in~\eqref{eqn:equi-disp-intuition}.
In Theorem~\ref{thm:suf-cond-equidisp} we present such a result.
We use the dominated convergence theorem to show that the limits commute under standard moment existence assumptions for a univariate simple MCP, which proves that it is equi-dispersed.

An implication of Theorem~\ref{thm:suf-cond-equidisp} is that the Poisson process is not the only equi-dispersed counting process.
In particular, Corollaries \ref{cor:equi-birth} and \ref{cor:equi-death} point out that the linear death process (or rather, the counting process associated with it) and the linear birth process, both extensively studied and used in applications, are also seen to be infinitesimally over-dispersed.
Nonetheless, these two processes are integrally under- and over-dispersed respectively. These integrated dispersion constraints are summarized in table~\ref{table:moments} and are a direct result of the well-know property that their increments follow binomial and negative binomial distributions respectively. Another implication, pointed out in Corollaries \ref{thm:suf-cond-mixed-equidisp} and \ref{cor:equi-mixed-Pois}, is that \emph{mixing} equi-dispersed MCPs with random variables does not alter dispersion.
The fact that both the mixed Poisson process and the birth process (with negative binomial increments) turn out to be integrally over-dispersed but infinitesimally equi-dispersed might be unexpected.

The moment existence conditions we use in our results concern the total number of events that a MCP $\{N(t)\}$ makes in an interval $[t,t+\bar h]$.
Specifically, define a stochastic bound of the infinitesimal rate function  $\lambda(N(s))$ conditional on $N(t) = n$ by
\begin{equation}\label{eqn:barLambda}
\bar\Lambda(t)=\sup_{t\le s\le t +\bar h}\lambda\big(N(s)\big).
\end{equation}
Here, we suppress the dependence of $\bar\Lambda(t)$ on $n$ and $\bar h$.
Now consider the following two properties:
\begin{itemize}
\item[\bf P1.] For each $t$ and $n$ there is some $\bar h>0$ such that $E[\bar \Lambda(t)]<\infty$.

\item[\bf P2.] For each $t$ and $n$ there is some $\bar h>0$ such that $V[\bar \Lambda(t)]<\infty$.
\end{itemize}
Properties P1 and P2 require that the MCP does not have explosive behavior, and in particular they hold for any \emph{uniform} MCP (i.e., a MCP for which $\q[n,k][][] \equiv \q[k][][]$) in which the jumps are bounded by some $k_0$ (i.e., when for all $k>k_0$, $\q[n,k][][] \equiv 0$).
P1 and P2 also hold for the simple, linear birth process and for the associated counting process to the simple, linear death process of Corollaries~\ref{cor:equi-birth} and~\ref{cor:equi-death}.

\begin{theorem}[sufficient condition for Markov infinitesimal equi-dispersion]\label{thm:suf-cond-equidisp}
Let $\{N(t)\}$ be a simple, time homogeneous, stable and conservative Markov counting process.
Supposing (P1), the infinitesimal mean is the same as the infinitesimal rate.
Supposing (P2), the infinitesimal variance is also the same as the infinitesimal rate, and therefore $\{N(t)\}$ is infinitesimally equi-dispersed.
\end{theorem}

\begin{proof}
Let $P(t)$ be a conditional Poisson process with event rate $\bar{\Lambda}(t)$ and work conditionally on $N(t)=n$ whenever the (potentially already conditional) expectation is taken over $\Delta N(t)$ or functions of it.
Then, since $\Delta N(t)$ is non-negative,
\begin{eqnarray}
\label{eqn:thm-suff-equi-I}
E[\Delta N(t)] = E[\Delta N(t)\Ind{\Delta N(t) > 0}] = E\Bigl[\Ind{\Delta N(t)=1} + \Delta N(t) \Ind{\Delta N(t) > 1}\Bigr].
\end{eqnarray}
Now, it is immediate that $E\big[\Ind{\Delta N(t)=1}\big]=\lambda(n)\dt + o(\dt)$.
Also, since $\{N(t)\}$ is simple, $\{\Delta N(t)\}$ is stochastically smaller than $\{\Delta P(t)\}$ and
\begin{eqnarray}
\nonumber E\big[\Delta N(t) \Ind{\Delta N(t) > 1}\big] &\leq& E\big[\Delta P(t) \Ind{\Delta P(t) > 1}\big]\\
\nonumber &=& E\Bigl[E\big[\Delta P(t) \Ind{\Delta P(t) > 1}|\bar{\Lambda}(t)\big]\Bigr].
\end{eqnarray}
Using~\eqref{eqn:thm-suff-equi-I} with $N(t)$ replaced by $P(t)$, noting also that $E\big[\Delta P(t)|\bar{\Lambda}(t)\big]=\dt \bar{\Lambda}(t)$ and $E\big[\Ind{\Delta P(t)=1}|\bar{\Lambda}(t)\big]=\dt \bar{\Lambda}(t) \exp\big\{-\dt \bar{\Lambda}(t)\big\}$, it follows that
\begin{eqnarray}
\nonumber E\big[\Delta N(t) \Ind{\Delta N(t) > 1}\big] &\leq& E\Big[\dt \bar{\Lambda}(t) - \dt \bar{\Lambda}(t) \exp\big\{-\dt \bar{\Lambda}(t)\big\}\Big]\\
\nonumber &=& E\Big[\dt \bar{\Lambda}(t)\Big(1 - \exp\{-\dt \bar{\Lambda}(t)\}\Big)\Big].%
\end{eqnarray}
It follows by dominated convergence, since $\bar{\lambda}\big(1 - \exp\{-\dt \bar{\lambda}\}\big) \leq \bar{\lambda}$ and by the assumption that $E[\bar{\Lambda}(t)]$ is finite (note that the distribution of $\bar{\Lambda}(t)$ depends on $\dtt$ and not $\dt$), that
\begin{eqnarray}
\nonumber \lim\limits_{\dt \downarrow 0} \frac{E\Big[\dt \bar{\Lambda}(t)\Big(1 - \exp\{-\dt \bar{\Lambda}(t)\}\Big)\Big]}{\dt} &=& E\bigg[\lim\limits_{\dt \downarrow 0} \bar{\Lambda}(t)\Big(1 - \exp\{-\dt \bar{\Lambda}(t)\}\Big)\bigg] = 0.%
\end{eqnarray}
Therefore, $E\big[\Delta N(t) \Ind{\Delta N(t) > 1}\big]=o(\dt)$ and $E[\Delta N(t)] = \lambda(n)\dt + o(\dt)$.

Similarly, replacing first by second moments, $E\big[(\Delta N(t))^2 \Ind{\Delta N(t) > 1}\big]=o(\dt)$ and $E[(\Delta N(t))^2] = \lambda(n)\dt + o(\dt)$, since
\begin{eqnarray}
\nonumber E\big[(\Delta N(t))^2 \Ind{\Delta N(t) > 1}\big] &\leq& E\big[(\Delta P(t))^2 \Ind{\Delta P(t) > 1}\big]\\
\nonumber &=& E\Big[E\big[(\Delta P(t))^2 \Ind{\Delta P(t) > 1}|\bar{\Lambda}(t)\big]\Big]\\
\nonumber &=& E\big[\dt \bar{\Lambda}(t) + \dt^2 \bar{\Lambda}^2(t) - \dt \bar{\Lambda}(t) \exp\{-\dt \bar{\Lambda}(t)\}\big]\\
\nonumber &\leq& E\big[2\dt^2 \bar{\Lambda}^2(t)\big]=o(\dt),%
\end{eqnarray}
where the last line follows by $1-\exp\{-x\} \leq x$ and $E[\bar{\Lambda}^2(t)]$ being finite.

Equi-dispersion follows from $V[\Delta N(t)] = E\big[(\Delta N(t))^2\big] - E[\Delta N(t)]^2 = \lambda(n)\dt + o(\dt)$, where  $E[\Delta N(t)]^2$ is $o(\dt)$ by stability of $\{N(t)\}$ which implies $\lambda(n) < \infty$ for all $n$.
\end{proof}

\begin{corollary}[necessary condition for Markov infinitesimal over-dispersion]\label{cor:nec-cond-overdisp}
Let $\{N(t)\}$ be a simple, time homogeneous, stable and conservative Markov counting process.
If, supposing (P1), the infinitesimal mean is not the same as the infinitesimal rate or if, supposing (P2), the infinitesimal variance is not the same as the infinitesimal rate or the process is not infinitesimally equi-dispersed, then $\q[n,j][][] \neq 0$ for some $j \geq 2$, i.e., $\{N(t)\}$ must be a compound process.
\end{corollary}

\begin{proof}
Otherwise, by Theorem~\ref{thm:suf-cond-equidisp}, the infinitesimal mean of $\{N(t)\}$ must coincide with the infinitesimal rate under P(1) and so must the infinitesimal variance under P(2) respectively, i.e. $\{N(t)\}$ must be infinitesimally equi-dispersed.
\end{proof}
\begin{corollary}[infinitesimal equi-dispersion of birth process]\label{cor:equi-birth}
A MCP with
$\q[n,1][][] = \beta n \Ind{n > 0}$
and $\q[n,k][][] = 0$ for $k > 1$ is a simple linear birth process for $\beta \in \R^+$ and is infinitesimally equi-dispersed.
\end{corollary}
\begin{proof}
This is a special cases of the multivariate Corollary~\ref{cor:equi-disp-birth-death} which is proved in Section~\ref{sec:notation-mult}.
\end{proof}

\begin{corollary}[infinitesimal equi-dispersion of death process]\label{cor:equi-death}
A MCP with $\q[n,1][][] = \delta (d_0 - n) \Ind{n < d_0}$ and $\q[n,k][][] = 0$ for $k > 1$ is the counting process associated with $\{\tilde N(t)\}$, a simple linear death process with initial population $d_0 \;\inNat$, for $\delta \in \R^+$ and is infinitesimally equi-dispersed.
\end{corollary}
\begin{proof}
This is a special cases of the multivariate Corollary~\ref{cor:equi-disp-birth-death} which is proved in Section~\ref{sec:notation-mult}.
\end{proof}

\begin{table}\label{table:moments}
\centering 
\begin{tabular}{@{}c | c@{} @{}c c@{}} 
\hline\hline 
\rule{0pt}{2.5ex}                     &   Poisson     &  Birth                               &   Death \\ [0.5ex] 
\hline 
\rule{0pt}{4ex} $E[\Delta N(t)|N(t)=n]$ & $\alpha \dt$ & $n(e^{\beta\dt}-1)$               & $(d_0 - n)(1 - e^{-\delta\dt})$ \\ [0ex] 
\rule{0pt}{4ex} $V[\Delta N(t)|N(t)=n]$ & $\alpha \dt$ & $ne^{\beta\dt}(e^{\beta\dt} - 1)$ & $(d_0 - n)(1 -e^{-\delta\dt})e^{-\delta\dt}$ \\ [0ex] 
\rule{0pt}{4ex} $D_N(n_0,t)$           & $1$           & $e^{\beta t}$                     & $e^{-\delta t}$ \\ [0ex] 
\rule{0pt}{4ex} $D_{dN}(n)$          & $1$           & $1$                               & $1$  \\ [0ex] 
\hline 
\hline 
\end{tabular}
\caption{ Increment mean, increment variance and dispersion indices of the time homogeneous (a) Poisson process (with individual infinitesimal rate $\alpha$), (b) linear birth process (with individual infinitesimal rate $\beta$ and initial population $n$) and (c) counting process associated with $\{\tilde N(t)\}$, a linear death process (with individual infinitesimal rate $\delta$ and initial population $d_0$).
}
\end{table}

\subsection{Dispersion of mixed Markov counting processes}\label{sec:nonMarkov}

The mixed Poisson process in Daley and Vere-jones \cite{daley2003}, which Snyder and Miller \cite{snyder1991} call  P\'{o}lya process, is a natural extension of the Poisson process where a \emph{mixing} random variable $M$ is used as the rate to define a Poisson process conditional on $M$.
An immediate result of this mixing is that the resulting process is integrally over-dispersed.
It is straightforward to generalize this notion to simple \emph{mixed MCPs}, where $\{N(t)\}$ is specified as a MCP conditional on $M$ with rate function $\Lambda(n) \equiv \lambda(n,M)$.
Mixed MCPs are non-Markovian but the measures of dispersion defined in~\eqref{eqn:DispIndex} and~\eqref{eqn:IntDisp} can still be computed and discussed.
For non-Markovian processes, conditioning on the entire past history in~\eqref{eqn:DispIndex} could also be considered.

Theorem~\ref{thm:suf-cond-mixed-equidisp} extends Theorem~\ref{thm:suf-cond-equidisp} by showing that conditions P1 and P2 ensure the equidispersion of simple mixed MCPs.
Here, the analogous definition to~\eqref{eqn:barLambda} is
$\bar\Lambda(t)=\sup_{t\le s \le t + \bar h}\Lambda\big(N(s)\big)$.
A proof of Theorem~\ref{thm:suf-cond-mixed-equidisp} is given in \ref{app:proof-overd-univ}.
In the context of mixed MCPs, P1 or P2 imply that $E[\Lambda(n)] < \infty$, so the tails of the additional randomness resulting from $M$ are required to be not too heavy.

\begin{theorem}[sufficient condition for mixed Markov infinitesimal equi-dispersion]\label{thm:suf-cond-mixed-equidisp}
Let $\{N(t)\}$ be a simple, time homogeneous, stable and conservative Markov counting process conditionally on a mixing random variable $M$.
Supposing (P1), the infinitesimal mean is the same as the average infinitesimal rate.
Supposing (P2), the infinitesimal variance is also the same as the average infinitesimal rate, and therefore $\{N(t)\}$ is infinitesimally equi-dispersed.
\end{theorem}

\begin{corollary}[infinitesimal equi-dispersion of mixed Poisson process]\label{cor:equi-mixed-Pois}
A conditional MCP with $\q[n,1,M][][] = M$ and $\q[n,k,M][][] = 0$ for $k > 1$ is a mixed Poisson process and is infinitesimally equi-dispersed if $E[M] < \infty$.
\end{corollary}
\begin{proof}
This follows directly from Theorem~\ref{thm:suf-cond-mixed-equidisp}.
\end{proof}


\section{Over-dispersed Univariate Markov Counting Processes\label{sec:OMCP}}

From Section~\ref{sec:notation}, we know that simple MCPs are equi-dispersed under standard moment conditions.
We therefore seek to generalize standard simple MCP models, to relax this dispersion constraint.
Our first approach is to investigate random time change, or subordination, which we show is equivalent to the inclusion of continuous-time noise in the rate function.
Then, in Section~\ref{sec:binom-beta}, we are led to consider a subtly different approach of defining an over-dispersed MCP via the limit of a sequence of processes in which discrete-time noise is used to modify the rate.

We know from Section~\ref{sec:notation} that introducing noise via a mixing random variable in the rate function does not alter the equi-dispersion of simple MCPs.
In other words, this additional variability disappears infinitesimally. This suggests considering more complex, alternative noise processes.
One possibility is to introduce some continuous-time process, say $\{\eta(t)\}$, in the rate function of the MCP.
Such constructions may be expected to give processes which are Markov conditional on $\{\eta(t)\}$ but not unconditionally.
Our approach is similar to that of \cite{marion2000} and \cite{varughese2008}; we propose defining a process by replacing $\lambda(n)$, the deterministic rate function of the original MCP, in Kolmogorov's backward differential system by the stochastic process
$\big\{\lambda\big(n,\eta(t)\big)\big\}$
(see~\ref{app:KBDS} for a formal definition).
However, by taking $\{\eta(t)\}$ to be a suitable white noise process, we differ from \cite{marion2000} and \cite{varughese2008} by constructing processes which will be shown to be unconditionally Markov.
The consideration of non-white noise is no doubt appropriate in some applications, but white noise provides a relatively simple extension to equi-dispersed processes controlled by a single intensity parameter.
Staying within the class of Markov processes also facilities both theoretical and numerical analysis of the resulting models.

The noise process $\{\eta(t)\}$ could enter $\lambda(n)$ additively or multiplicatively.
Given the non-negativity constraint on the infinitesimal rate functions, multiplicative non-negative noise is a simple and convenient choice.
We refer to white noise, $\{\xi(t)\} \equiv \{dL(t)/dt\}$, as the derivative of an \emph{integrated noise} process $\{L(t)\}$ which has stationary independent increments.
Note that we do not necessarily require that the mean of $L(t)$ is zero.
Although $\{\xi(t)\}$ may not exist, in the sense that $\{L(t)\}$ may not have differentiable sample paths,  $\{\xi(t)\}$ can nevertheless be given formal meaning \citep{karlin1975,breto2009-aoas}.
Restricting $\{\xi(t)\}$ to non-negative white noise, the family of increasing L\'{e}vy processes provides a rich class from which to choose the integrated noise $\{L(t)\}$.
Multiplicative unbiased noise is achieved by requiring $E[L(t)]=t$, in which case  $\lim_{\dt \downarrow 0}E[\Delta L(t)]\big/\dt = 1$.

From an alternative perspective, in the context of the general theory of Markov processes, random time change or subordination of an initial process is a well established tool to obtain new processes.
Following Sato \cite{sato1999}, let $\{M(t)\}$ (the \emph{directing} process) be a temporally homogeneous Markov process and $\{L(t)\}$ (the \emph{subordinator}) be an increasing L\'{e}vy process.
Any temporally homogeneous Markov process $\{N(t)\}$ identical in law to $\{M \circ L(t)\} \equiv \big\{M\big(L(t)\big)\big\}$ is said to be \emph{subordinate} to $\{M(t)\}$ by the subordinator $\{L(t)\}$.

Theorem~\ref{thm:KBDS} below (proved in \ref{app:KBDS}) formally states that subordinate processes to simple (and hence equi-dispersed) MCPs are equivalent to solutions of L\'{e}vy-driven stochastic differential equations resulting from introducing unbiased multiplicative L\'{e}vy white-noise in the deterministic Kolmogorov backward differential system of the directing process.
This gives us a licence to interpret noise on the rate of an MCP as subordination of the MCP to a L\'{e}vy process.
In Subsection~\ref{sec:gamma-MCPs}, we obtain exact results when investigating concrete examples of over-dispersion by exploiting this connection between gamma white noise in the rates and  gamma subordinators.
The general arguments of \ref{app:KBDS} and the particular processes of Subsection~\ref{sec:gamma-MCPs} may both be of interest to the reader, but to preserve the flow of the main themes of this paper we have chosen to defer the technical details involved in the link between subordination and stochastic rates to an appendix.

\begin{theorem}[L\'{e}vy white noise and subordination] \label{thm:KBDS}
Consider the simple, time homogeneous, stable and conservative Markov counting process $\{M_{\lambda}(t)\}$ defined by the rate function $\lambda(m)$.
Let $\{L(t)\}$ be a non-decreasing, L\'{e}vy process with $L(0)=0$ and $E[L(t)]=t$.
Let $\{M_{\lambda \xi}(t)\}$ be the process resulting from introducing unbiased, non-negative, multiplicative, L\'{e}vy white-noise $\{\xi(t)\} \equiv \{dL(t)/dt\}$ in the rate of $\{M_{\lambda}(t)\}$, defined as the solution to the L\'{e}vy-driven Kolmogorov backward differential system in~\eqref{eqn:stoch-KBDS}.
Then, if this solution exits and is unique,
\[
M_{\lambda \xi}(t) \sim M_{\lambda} \circ L(t) \sim M_{\lambda}\Big({\textstyle \int_{0}^{t} } \xi(u)\;du\Big).
\]
\end{theorem}

\subsection{Subordinate processes to simple MCPs by gamma subordinators}\label{sec:gamma-MCPs}

A convenient candidate for non-negative continuous-time noise is gamma noise.
In this case, the integrated noise process $L(t) = \Gamma(t)$ is a Gamma process defined to have independent, stationary increments with $\Gamma(t)-\Gamma(s)\sim \mbox{Gamma}\,([t-s]/\tau,\tau)$.
Here, $\mbox{Gamma}\,(\alpha,\beta)$ is the gamma distribution with mean $\alpha\beta$ and variance $\alpha\beta^2$.
In Subsections~\ref{sec:pois-gamma}--\ref{sec:negbin-gamma}, we study the inclusion of gamma noise in the rates of the Poisson process, linear birth process and linear death process, each of which have been shown to be equi-dispersed in Section~\ref{sec:notation}.

Following the convention for naming of subordinate 
processes \cite{sato1999}, we will place the name of the original process first, followed by the name of the driving subordinating noise.
We have chosen to study in detail the Poisson, linear birth and linear death processes because they are basic blocks widely used to build more complex, multi-process models, such as compartmental models used in population dynamics and queuing networks in engineering.
What makes these three processes fundamental is that they capture in the simplest way, i.e. linearly, the most common possibilities in real applications.
Namely, events that by occurring ``kill'' the potential for future events (death process, or negative feedback); events that ``reproduce'' meaning that their occurrence fuels that of future events (birth process, or positive feedback); and events which occur independently of the events which have already happened (Poisson, or immigration process, or no feedback).
However, our approach could be extended to other processes that might be of interest.

For these processes we provide three results: their first two moments about the mean, which show they are indeed over-dispersed; the distribution of the counting process, which allows for exact, direct simulation of the counting process; and a closed form for the infinitesimal probabilities, which fully characterize the processes and may be used for exact simulation of the event times of the point process and for indirect, exact simulation of the counting process by aggregation.

Since, as shown in Section~\ref{sec:notation-mult}, multivariate processes built upon univariate processes retain the dispersion constraints of the latter, constructing over-dispersed multivariate processes, a conceptually more complex task, can be achieved using the provided infinitesimal probabilities of over-dispersed univariate processes as building blocks, the same way it is routinely done with equi-dispersed processes. This highlights the relevance of the univariate results in this section.

\subsubsection{The Poisson gamma process}\label{sec:pois-gamma}
We construct an over-dispersed Poisson process.
This is a special case of the general compound Poisson process \citep{daley2003}, which can be constructed as independent jumps from an arbitrary distribution occurring at the times of a Poisson process.
Our alternative construction, derived through introducing white noise on the rate, has an advantage that it can be applied (as we show) not just to Poisson processes but to more general univariate and multivariate processes.

\begin{proposition}[Poisson gamma process]\label{pro:Pois-gamma-proc}
Let $\{M(t)\}$ be a MCP with $\q[n,1][][] = \alpha$ and $\q[n,k][][] = 0$, i.e. a time homogeneous Poisson process with rate $\alpha$.
Introducing continuous-time gamma noise $\{\xi(t)\} \equiv \{d\Gamma(t)/dt\}$ where $\Gamma(t) \sim \text{Gamma}\,(t/\tau,\tau)$ defines $\{N(t)\}$, a compound infinitesimally over-dispersed MCP with increment probabilities for $k \;\inNat$
\begin{eqnarray*}
P(\Delta N(t)=k|N(t)=n) &=& \frac{G\left(\tau^{-1}\dt + k\right)}{k!G(\tau^{-1} \dt)}p^{\tau^{-1}\dt}\left(1-p\right)^k,
\end{eqnarray*}
where $\omega = \tau \alpha$, $p = \left(1 + \omega \right)^{-1}$ and we use $G$ for the gamma function.
The infinitesimal probabilities are
\begin{eqnarray}
\nonumber \q[n,k][][] &=& \tau^{-1}\frac{\left(1-p\right)^k}{k}.
\end{eqnarray}
for $k \geq 1$ and zero otherwise.
The infinitesimal rate is
\begin{eqnarray}
\nonumber \lambda(n) &=& \tau^{-1}\log (p^{-1})
\end{eqnarray}
The infinitesimal moments are
$\mu_{dN} = \alpha$ and $\sigma^2_{dN} = (1+ \tau \alpha ) \mu_{dN}$ with dispersion $D_{dN}(n)=1 + \tau \alpha$.
\end{proposition}

\subsubsection{The binomial gamma process} \label{sec:binom-gamma}
Here, we consider multiplicative gamma noise on the rate of a linear death process.
This process has been proposed as a model for biological populations \citep{breto2009-aoas}, although it was defined as the limit of discrete-time stochastic processes rather than as the solution to the L\'{e}vy-driven Kolmogorov differential system of~\eqref{eqn:stoch-KBDS}.
It is standard to define death processes as decreasing processes, however our general framework has been for counting processes which are necessarily increasing.
To resolve this minor point, we will use the following notation.
For some positive integer $d_0$, define $\{\tilde A(t)\} \equiv \{\max \{d_0 - A(t), \,0 \}\}$ where $\{A(t)\}$ is a MCP.
The tilde represents then a transformation which, when applied to a MCP, defines a non-increasing Markov process which may be thought of as the number of individuals still alive by time $t$ out of the initial $d_0$ when the MCP $\{A(t)\}$ counts the number of deaths.

\begin{proposition}[binomial gamma process]\label{pro:binom-gamma-proc}
Let $\{M(t)\}$ be a MCP with $\q[m,1][][] = (d_0 - m) \Ind{m < d_0}$ and $\q[m,k][][] = 0$ for $k > 1$, i.e. the counting process associated with a linear death process $\{\tilde{M}(t)\}$ with individual death rate $\delta \in R^+$ and initial population size $d_0 \;\inNat$.
Introducing continuous-time gamma noise $\{\xi(t)\} \equiv \{d\Gamma(t)/dt\}$ where $\Gamma(t) \sim \text{Gamma}\,(t/\tau,\tau)$
defines $\{N(t)\}$, a compound infinitesimally over-dispersed MCP with increment probabilities
\begin{eqnarray*}
P(\Delta N=k| N(t)=n) &=&
    {\tilde{n} \choose k}
    \sum\limits_{j=0}^{k}{
    {k \choose j}
    }(-1)^{k-j}
    \Bigl(1  + \delta \tau(\tilde{n} - j)\Bigr)^{-\dt\tau^{-1}}
\end{eqnarray*}
and infinitesimal probabilities
\begin{eqnarray}
\nonumber \q[n,k][][] &=& {\tilde{n} \choose k} \sum\limits_{j=0}^{k}{{k \choose j}}
(-1)^{k - j + 1}\tau^{-1} \ln{\bigl(1  + \delta \tau (\tilde{n} - j)\bigr)}
\end{eqnarray}
for $n < d_0$ and $k \in \{0,\dots,\tilde n\}$, and zero otherwise.
Here, $\tilde n \equiv d_0 - n$.
The infinitesimal rate is
\begin{eqnarray}
\nonumber \lambda(n) &=& \tau^{-1} \ln{\bigl(1  + \delta \tau\tilde{n}\bigr)}
\end{eqnarray}
The infinitesimal moments and dispersion are
\bean
\mu_{dN} &=& \tilde{n}\tau^{-1}\ln(1+\delta\tau)\\
\sigma^2_{dN} &=& \mu_{dN} + \tilde{n}\tau^{-1}\Biggl[\bigl(\tilde{n} - 1\bigr)\ln\biggl(\frac{(1+\delta\tau)^2}{1+2\delta\tau}\biggr)\Biggr]\\
D_{dN}(n) &=& 1 + (\tilde{n} - 1)\Biggl[\frac{2\ln(1+\delta\tau)- \ln(1+2\delta\tau)}{\ln(1+\delta\tau)}\Biggr].
\eean
Hence, $\{N(t)\}$ is infinitesimally over-dispersed for $\tilde{n} > 1$ and equi-dispersed for $\tilde{n} = 1$.
\end{proposition}

\subsubsection{The Negative Binomial Gamma Process}\label{sec:negbin-gamma}
Unlike for the death process, when introducing gamma noise to the birth process we are only able to show existence of moments imposing a restriction on the parameter space.
In particular, the birth rate of the original process imposes an upper bound on the over-dispersion.
When this restriction does not hold, the moments of the resulting process do not exist, and hence our dispersion index is not defined.
We include the derivations with gamma noise for consistency with the Poisson and death process.
Considering a common subordinator for all three processes has the advantage that it leads naturally to the multivariate situations in Section~\ref{sec:notation-mult}, in which over-dispersed univariate processes are combined to construct multivariate models.
It would be possible to use other subordinators, such as the inverse Gaussian process, for which the moment generating function is available in closed form.

\begin{proposition}[negative binomial gamma process]\label{pro:negbin-gamma-proc}
Let $\{M(t)\}$ be a MCP with $\q[m,1][][] = \beta m \Ind{m > 0}$ and $\q[m,k][][] = 0$ for $k > 1$, i.e. a linear birth process for $\beta \in \R^+$.
Introducing continuous-time gamma noise $\{\xi(t)\} \equiv \{d\Gamma(t)/dt\}$ where $\Gamma(t) \sim \text{Gamma}\,(t/\tau,\tau)$
defines $\{N(t)\}$, a compound infinitesimally over-dispersed MCP with increment probabilities for $k, n\;\inNat$
\begin{eqnarray*}
P(\Delta N=k| N(t)=n) &=&
    {n + k - 1 \choose k}
    \sum\limits_{j=0}^{k}{
    {k \choose j}
    }(-1)^{k-j}
    \Bigl(1  + \beta \tau(n + k - j)\Bigr)^{-\dt\tau^{-1}}
\end{eqnarray*}
and infinitesimal probabilities
\begin{eqnarray*}
\nonumber \q[n,k][][] &=& {n + k - 1 \choose k} \sum\limits_{j=0}^{k}{{k \choose j}}
(-1)^{k - j + 1}\tau^{-1} \ln{\bigl(1  + \beta \tau (n + k - j)\bigr)}
\end{eqnarray*}
and zero otherwise.
The infinitesimal rate is
\begin{eqnarray*}
\nonumber \lambda(n) &=& \tau^{-1} \ln{\bigl(1  + \beta \tau n\bigr)}
\end{eqnarray*}
For $2\beta \tau < 1$, the infinitesimal moments and dispersion are
\bean
\mu_{dN} &=& n\tau^{-1}\ln\Bigl(\frac{1}{1-\beta\tau}\Bigr)\\
\sigma^2_{dN} &=& \mu_{dN} + n\tau^{-1}\Bigl[\biggl(n - 1\bigr)\ln\biggl(\frac{(1-\beta\tau)^2}{1-2\beta\tau}\biggr)\Biggr]\\
D_{dN}(n) &=& 1 + (n - 1)\Biggl[\frac{2\ln(1-\beta\tau)- \ln(1-2\beta\tau)}{-\ln(1-\beta\tau)}\Biggr].
\eean
Hence, $\{N(t)\}$ is infinitesimally over-dispersed for $n > 1$ and equi-dispersed for $n = 1$.
\end{proposition}

\subsubsection{The binomial beta process} \label{sec:binom-beta}

The infinitesimal moments of the binomial gamma process are a non-linear system of two equations which, to obtain a desired mean and variance, needs to be solved numerically for $\delta$ and $\tau$ with which the process is actually parameterized.
A moment-based parameterization allows to easily change the variability (via the variance) for a fixed location (fixing the mean), allowing for an easy interpretation of parameter changes.
Other parameterizations may require changes in several parameters to achieve the same goal.
In the context of counting processes, such parameterization has the additional advantage that it permits a direct and straightforward comparison with analogous stochastic differential equations.

As an alternative, the binomial beta process, as defined below, can be easily parameterized in terms of the infinitesimal moments.
Instead of introducing continuous-time noise to the rates, we consider introducing it directly to the event probabilities of the death binomial process.
Since the constraint on probabilities is the unit interval we need to consider an alternative to gamma noise.
An obvious alternative would be beta noise.
The construction of a beta process as a process with beta independent increments is not, however, straightforward \cite{hjort1990}.
We therefore consider in this section an alternative construction of compound processes based on noise introduction consisting on taking limits of discrete-time processes \cite{breto2009-aoas}.
For this construction, let $\{\Pi_0,\Pi_1, \Pi_2, \ldots\}$ be an infinite collection of independent and identically distributed random variables and, for each fixed $\dt > 0$, define the continuous-time process $\{\Pi_{\dt}(t)\}$ by $\Pi_{\dt}(t) \equiv \Pi_i$ for $t \in [i \dt, (i+1)\dt)$.
For each $h>0$, we can conditionally define a death process with time-varying rate $\{\Pi_{\dt}(t)\}$.
Integrating out over the distribution of $\{\Pi_{\dt}(t)\}$ results in an (unconditional) process for each $h>0$, and the resulting infinitesimal probabilities in the limit as  $\dt$ tends to zero define a Markov process.
As proved in Proposition~\ref{pro:binom-beta-proc} below, when $\Pi_i \sim \mbox{Beta}(\alpha, \beta)$, this construction defines the infinitesimal probabilities of a new process, which we call the beta binomial process. Here $\mbox{Beta}(\alpha, \beta)$ is a beta distribution with mean $\alpha / (\alpha + \beta)$ and variance $\alpha \beta / \Big((\alpha + \beta)^2 (\alpha + \beta + 1)\Big)$.

In spite of the more convenient parameterization, the integrated increment probabilities of the beta binomial process are not obtained as a byproduct as in Sections~\ref{sec:pois-gamma}--\ref{sec:negbin-gamma}.
Hence, exact simulation of the counts is only possible (with the present results) by aggregation from exact event time simulation based on the provided, defining infinitesimal probabilities.
As in Section~\ref{sec:binom-gamma}, let $\tilde{M}(t) = d_0 - M(t)$.
\begin{proposition}[binomial beta process]\label{pro:binom-beta-proc}
Let $\{M(t)\}$ be a MCP with $\q[m,1][][] = \delta (d_0 - m) \Ind{m < d_0}$ and $\q[m,k][][] = 0$ for $k > 1$, i.e. the counting process associated with a linear death process $\{\tilde{M}(t)\}$ with individual death rate $\delta \in R^+$ and initial population size $d_0 \;\inNat$.
Let $\c = \frac{\tilde{n} - 1}{\omega} - 1$ for $\tilde{n} > 1$ and $c \in \R^+$ otherwise with $0 < \omega < \tilde{n} - 1$.
Introducing continuous-time discretized beta noise $\{\Pi_{\dt}(t)\}$ in the death probabilities over an interval $[t, t + \dt]$, where $\Pi(t)_{\dt} \sim \mbox{Beta}\Big(\;\c(1-e^{-\delta \dt}),\; \c e^{-\delta \dt}\Big)$, and taking the limit $\dt\to 0$, defines $\{N(t)\}$,
a compound infinitesimally over-dispersed MCP with infinitesimal probabilities for $n < d_0$ and $k \in \{0,\dots,\tilde{n}\}$
\begin{eqnarray*}
\q[n,k][][] &=& {\tilde{n} \choose k} \frac{G(k)G(\c + \tilde{n} - k)} {G(\c + \tilde{n})}\c\delta
\end{eqnarray*}
and zero otherwise.
Here, $\tilde n \equiv d_0 - n$.
The infinitesimal moments are $\mu_{dN} = \tilde{n}\delta$ and, for $\tilde{n} > 1$, $\sigma^2_{dN} = (1+ \omega) \mu_{dN}$ with dispersion $D_{dN}(n)= 1 + \omega$ so that $\{N(t)\}$ is infinitesimally over-dispersed. If $\tilde{n} = 1$,  $\mu_{dN} = \sigma^2_{dN}= \tilde{n}\delta$ and $\{N(t)\}$ is equi-dispersed.
\end{proposition}


\section{Markov Counting Systems and Multivariate Dispersion}\label{sec:notation-mult}

Although univariate counting processes have applications in their own right, more complex scenarios will in general require multivariate processes.
We begin this section by defining a multivariate extension of MCPs and generalizing the dispersion indices of Section~\ref{sec:notation} to such multivariate processes, after introducing some necessary additional notation.
The study of integrated moments of multivariate processes has given rise to a substantial literature, often considering truncation of the moment generating function \citep{matis2000}.
Even though these integrated moments have been analyzed, we are not aware of any attempt to formally treat dispersion of a multivariate process.
In this section, we do this from both an integrated and infinitesimal perspective, focusing on the latter.
Then, we establish sufficiency and necessity of simpleness and compoundness for equi- and over-dispersion respectively for multivariate processes, generalizing the results of  Theorem~\ref{thm:suf-cond-equidisp} and Corollary~\ref{cor:nec-cond-overdisp}.
We illustrate this generalization by establishing equi-dispersion of multivariate birth-death processes.
This result will lead us in Section~\ref{sec:OMCS} to show that the univariate over-dispersed processes of Section~\ref{sec:OMCP}, together with other standard univariate processes, can be used as \emph{building blocks} to construct multivariate processes which retain the dispersion of the univariate blocks.
We will illustrate this point by constructing an over-dispersed multivariate birth-death process.
The integrated properties of the resulting new multivariate process will, however, generally be different from those of the univariate blocks.
This is further evidence for the conceptual simplicity arising from a focus on infinitesimal dispersion.
Henceforth, as earlier, we will mean infinitesimal dispersion when we simply talk about dispersion.
In this ``building block'' approach, it is straightforward to consider simultaneous events in any given block (e.g., blocks of the form of the over-dispersed processes of Section~\ref{sec:OMCP}). However, this univariate-flavored approach is not amenable to simultaneous events among different blocks.
In the context of the processes of Section~\ref{sec:OMCP}, such an approach would require considering dependencies among different white noises, which is outside the scope of this article.

In order to define a multivariate extension of MCPs, consider a finite family of counting processes $\big\{\{N_{ij}(t)\}:t \in \R^+, \; i \neq j, \;i \in \MultI, \; j \in \MultII\big\}$ with starting conditions $N_{ij}(0) = 0$.
We will refer to the multivariate, vector-valued process formed by such a family of processes as $\{\bm{N}(t)\}$, and we will use bold letters for multivariate processes.
The family member $\{N_{ij}(t)\}$ counts events of the \emph{$ij$-type}.
Each event type can be interpreted as a transition from one state (which we call the initial condition) to another (called the final condition).
The set $\MultI$ (and $\MultII$) consists of all possible initial (and final) conditions, and the set of all possible conditions we call $\MultSet \equiv \MultI \bigcup \MultII$.
As an illustration, consider the simple, linear birth process, where sets $\MultI$ and $\MultII$ could be defined as $\MultI = \{\textit{``unborn''}\}$ and $\MultII = \{\textit{``alive''}\}$ and $N_{\textit{``alive''}\textit{``unborn''}}(t) = 0$.
To ease notation in the subindices, we will assume in what follows that the elements of $\MultI$ and $\MultII$ are relabeled so that $\MultSet = \{1, \ldots ,\MultCar\}$ with $\MultCar$ being the cardinality of $\MultSet$.
Let a \emph{Markov counting system} (MCS), with associated multivariate counting process $\{\bm{N}(t)\}$, be the integer-valued Markov process $\{\bm{X}(t)\}  \equiv \big\{\big(X_1(t),\dots,X_{\MultCar}(t)\big)\big\}$ for $\bm{X}(0) \in \Z^{\MultCar}$
defined by ``conservation of mass'' identities
\begin{eqnarray}
\label{eqn:mass-conservation}
X_{\MultInd}(t) = X_{\MultInd}(0)+\sum_{i \neq \MultInd} N_{i\MultInd}(t)-\sum_{j \neq \MultInd} N_{\MultInd j}(t),
\end{eqnarray}
for $\MultInd \in 1,\ldots,\MultCar$ so that $\bm{X}(t) \in \Z^{\MultCar}$
The infinitesimal probabilities of the MCS $\{\bm{X}(t)\}$ are defined, via~\eqref{eqn:mass-conservation}, by the infinitesimal probabilities of $\{\bm{N}(t)\}$ which are in turn assumed to be a function of $\bm{X}(t)$.
Specifically, we define the non-zero infinitesimal probabilities to be
\begin{eqnarray}
\label{eqn:MCS} \q[\bm{x},k][ij][]  &\equiv& \lim\limits_{\dt \downarrow 0} \frac{P(\Delta N_{ij}=k \mbox{ and } \Delta N_{lm}=0 \mbox{ for } (l,m) \neq (i,j) |\bm{X}(t)=\bm{x})}{\dt}
\end{eqnarray}
for $i \neq j$, $k \inNatII$ and $\bm{x} \in \Z^{\MultCar}$.
Analogously to Section~\ref{sec:notation}, define the rate function of this MCS to be
\begin{eqnarray*}
\lambda(\bm{x}) \equiv \lim\limits_{\dt \downarrow 0} \frac{1 - P(\Delta N_{ij}=0\mbox{ for all $(i,j)$}|\bm{X}(t)=\bm{x})}{\dt}.\\
\end{eqnarray*}
Again we restrict ourselves to stable and conservative processes so that $\lambda(\bm{x}) = \sum_{ijk} \q[\bm{x},k][ij][] < \infty$ for all $\bm{x}$, with the sum being over $i,j \in \MultSet$, $i\neq j$, $k\ge 1$.
As motivated above, we have restricted the processes under consideration to those where simultaneous events are possible but only for a given process at any event time.
This is a natural choice since our ultimate goal is to construct multivariate over-dispersed processes using univariate ones as building blocks.

$\{\bm{X}(t)\}$ can be interpreted as a compartment model \cite{jacquez1996} with $\MultCar$ compartments and where $N_{ij}(t)$ counts the direct flow from compartment $i$ to $j$.
Equivalently, it can also be interpreted as a queuing network.
The conservation equations in~\eqref{eqn:mass-conservation} require that the sum of the $\MultCar$ processes $\{X_{\MultInd}(t)\}$ remains unchanged at any given moment in time.
However, they do not require that the $\{X_{\MultInd}(t)\}$ processes be bounded, as would occur in a standard compartment model with fixed total population size in which compartment counts are necessarily non-negative.
This is achieved by allowing at least some $\{X_{\MultInd}(t)\}$ to be negative.
As examples, we consider the possibility of birth, immigration and death events.
To model births or immigration or both into compartment $r$, we can require by construction that there is some $s$ for which  $X_s(t)=-N_{sr}(t)$.
Then,  $X_s(t)$ keeps track of the negative count of births or immigration which enter $r$ (and so lead to an increase in $X_r(t)$).
We call $X_s(t)$ a \emph{source} process.
One can similarly define a \emph{sink} process to count death events.

We slightly generalize the definition of dispersion indices of univariate counting processes as follows.
We will refer to
\begin{eqnarray}
D_{d\bm{X}}^{ij}(\bm{x}) & \equiv& \frac{\lim_{\dt \downarrow 0}\dt^{-1}V[N_{ij}(t+\dt)-N_{ij}(t)|\bm{X}(t)=\bm{x}]}{\lim_{\dt \downarrow 0} \dt^{-1}E[N_{ij}(t+\dt)-N_{ij}(t)|\bm{X}(t)=\bm{x}]}
\equiv \frac{\sigma^{2\,\,ij}_{d\bm{X}}(\bm{x})}{\mu_{d\bm{X}}^{ij}(\bm{x})}
\label{eqn:DispIndex-Mult}
\end{eqnarray}
as the infinitesimal dispersion index of $\{N_{ij}(t)\}$, the \emph{$ij$-marginal} counting process associated to the MCS $\{\bm{X}(t)\}$.
Since the infinitesimal moments in (\ref{eqn:DispIndex-Mult}) are a function of $\bm{x}$ and are directly related to the increments of $\{\bm{X}(t)\}$, we refer to the collection $\{D_{d\bm{X}}^{ij}(\bm{x}): i \in \MultI, \, j \in \MultII\}$ as the dispersion of the multivariate MCS $\{\bm{X}(t)\}$.
Following our approach of studying simultaneous events only on single transitions (equivalent here to noise processes $\{\xi_{ij}(t)\}$ which are independent for $i\neq j$), we do not need to consider infinitesimal covariance here.
We correspondingly define the integrated counterpart of~\eqref{eqn:DispIndex-Mult} as
\begin{eqnarray}
\nonumber D_{\bm{X}}^{ij}(\bm{x}_0,t) &\equiv& \frac{V[N_{ij}(t)-N_{ij}(0)|\bm{X}(0)=\bm{x}_0]}{E[N_{ij}(t)-N_{ij}(0)|\bm{X}(0)=\bm{x}_0]}.
\end{eqnarray}
We will now say that an $ij$-marginal counting process $\{N_{ij}(t)\}$ is (integrally) over-dispersed if $D^{ij}_{d\bm{X}} > 1 \, (D^{ij}_{\bm{X}} > 1)$ for $i,j \in \MultSet$.
We will also say that a MCS $\{\bm{X}(t)\}$ is (integrally) over-dispersed if $D^{ij}_{d\bm{X}} > 1 \, (D^{ij}_{\bm{X}} > 1)$ for at least some $i,j \in \MultSet$.
As in Section~\ref{sec:notation}, we define under- and equi-dispersion accordingly.

We now prove Theorem~\ref{thm:suf-cond-mult-build-block} below, which gives the infinitesimal mean and variance of the $ij$-marginal counting processes associated with a MCS.
Similarly to~\eqref{eqn:DispIndex} in the univariate section, these two moments are the denominator and numerator of~\eqref{eqn:DispIndex-Mult} respectively.
Note that these moments are now conditional on $\bm{X}(t)=\bm{x}$ because we have allowed in the definition of a MCS in~\eqref{eqn:MCS} for possible dependencies between different $ij$-marginal processes.
We use Theorem~\ref{thm:suf-cond-mult-build-block} to prove the corollaries which complete this section regarding both sufficient and necessary conditions for dispersion of MCSs.
We will also use Theorem~\ref{thm:suf-cond-mult-build-block} in Section~\ref{sec:OMCS} to prove that dispersion of MCSs constructed using univariate blocks retain the dispersion properties of the blocks.
The moment existence conditions we use concern now the total number of events that $\{\bm{N}(t)\}$, associated to $\{\bm{X}(t)\}$ through~\eqref{eqn:mass-conservation}, makes in an interval $[t, t + \bar h]$ and the size of simultaneous events in this same interval for each $ij$ marginal.
Analogously to Section~\ref{sec:notation}, define now
\[
\bar\Lambda_{ij}(t) \equiv \sup_{t\le s\le t+\bar h}\lambda\big(\bm{X}(s)\big).
\]
Note that we use the same bound for all $ij$-marginal processes.
Since we are allowing the possibility of compound processes, we also need to define a stochastic bound conditional on $\bm{X}(t)=\bm{x}$ for the size of simultaneous events
\[
\bar{Z}_{ij}(t) \equiv \sup\limits_{t \leq s \leq t+\dtt}{dN_{ij}(s)}.
\]
Again, we suppress the dependence of $\bar\Lambda_{ij}(t)$ and of $\bar Z_{ij}(t)$ on $\bm{x}$ and $\bar h$.
Now consider:
\begin{itemize}
\item[\bf P3.] For each $t$, $\bm{x}$ and $i \ne j$ there is some $\bar h>0$ such that $E[\bar Z_{ij}(t)\bar \Lambda_{ij}(t)]<\infty$.

\item[\bf P4.] For each $t$, $\bm{x}$ and $i \ne j$ there is some $\bar h>0$ such that $V[\bar Z_{ij}(t)\bar \Lambda_{ij}(t)]<\infty$.
\end{itemize}
Properties P3 and P4 again require that the $ij$-marginal counting processes do not have explosive behavior, and in particular they hold for simple birth-death processes with linear birth and death rates, as shown in Corollary~\ref{cor:equi-disp-birth-death}.

\begin{theorem}[infinitesimal moments of a MCS]
\label{thm:suf-cond-mult-build-block}
Let $\{\bm{X}(t)\}$ be a time homogeneous, stable and conservative Markov counting system with associated multivariate counting process $\{\bm{N}(t)\}$ defined by~\eqref{eqn:mass-conservation} and~\eqref{eqn:MCS}.
Supposing (P3), the infinitesimal mean of $\{N_{ij}(t)\}$ is the first moment of its jump distribution, i.e. $\mu_{d\bm{X}}^{ij}(\bm{x})=\sum_k k \q[\bm{x},k][ij][]$. Supposing (P4), its infinitesimal variance is the second moment of its jump distribution, i.e. $\sigma^{2\,\,ij}_{d\bm{X}}(\bm{x})=\sum_k k^2 \q[\bm{x},k][ij][]$.
\end{theorem}
\begin{proof}
Let $\bar{N}_{ij}(t)$ be a conditional compound Poisson process with event rate $\bar{\Lambda}(t) \equiv \bar{\Lambda}_{ij}(t)$ and degenerate jump distribution with mass one at $\bar{Z}(t) \equiv \bar{Z}_{ij}(t)$. Work conditionally on $\bm{X}(t)=\bm{x}$ whenever the (potentially already conditional) expectation is taken over $\Delta N_{ij}(t)$ or functions of it.
Let $S$ be the event of exactly one single jump (of size one or more) in any of the $\{N_{ij}(t): i \ne j\}$ counting processes occurring in $[t,t+\dt]$, as in Lemma~\ref{lem:prob-Scomp-little-oh}.
Then,
\begin{eqnarray}
\label{eqn:suf-cond-multi-build-block-I}E[\Delta N_{ij}(t)] &=& E[\Delta N_{ij}(t)\Ind{S}] + E[\Delta N_{ij}(t)\Ind{S^c}].
\end{eqnarray}
Unlike in Theorem~\ref{thm:suf-cond-equidisp}, the term corresponding to one single jump is not immediate and requires Lemma~\ref{lem:prob-Scomp-little-oh}.
Letting $S_{ij}$ be the event of exactly one single jump (of size one or more) in the $ij$ counting process occurring in $[t,t+\dt]$,
\begin{eqnarray*}
E[\Delta N_{ij}(t)\Ind{S}] &=& E[\Delta N_{ij}|S_{ij}] P(S_{ij}|S) P(S)\\
                         &=& \sum\limits_{k} k \frac{\q[\bm{x},k][ij][]}{\sum\limits_k \q[\bm{x},k][ij][]}\frac{\sum\limits_k \q[\bm{x},k][ij][]}{\sum\limits_{i,j,k}\q[\bm{x},k][ij][]} \times \Bigl[\dt \sum\limits_{i,j,k}\q[\bm{x},k][ij][] + o(\dt)\Bigr]\\
                         &=& \dt \sum\limits_{k} k \q[\bm{x},k][ij][] + o(\dt)\\
\end{eqnarray*}
Analogously to Theorem~\ref{thm:suf-cond-equidisp}, we proceed to bound the second term to show the desired result. Since $\Delta N_{ij}(t)$ is stochastically smaller than $\Delta \bar{N}_{ij}(t)$,
\begin{eqnarray}
\nonumber E[\Delta N_{ij}(t) \Ind{S^c}] &\leq& E[\Delta \bar{N}_{ij}(t) \Ind{S^c}]\\
\nonumber &=& E[E[\Delta \bar{N}_{ij}(t) \Ind{S^c}|\bar{\Lambda}(t),\bar{Z}(t)]]%
\end{eqnarray}
Using~\eqref{eqn:suf-cond-multi-build-block-I} with $N_{ij}(t)$ replaced by $\bar{N}_{ij}(t)$ and since $E[\Delta \bar{N}_{ij}(t)|\bar{\Lambda}(t),\bar{Z}(t)]=\bar{Z}(t)\dt \bar{\Lambda}(t)$ and $E[\Delta \bar{N}_{ij}(t)\Ind{S}|\bar{\Lambda}(t),\bar{Z}(t)]=\bar{Z}(t) \dt \bar{\Lambda}(t) \exp\{-\dt \bar{\Lambda}(t)\}$, it follows that
\begin{eqnarray}
\nonumber E[\Delta N_{ij}(t) \Ind{S^c}] &\leq& E[\bar{Z}(t)\dt \bar{\Lambda}(t) - \bar{Z}(t)\dt \bar{\Lambda}(t) \exp\{-\dt \bar{\Lambda}(t)\}]\\
\nonumber &=& E[\bar{Z}(t) \dt \bar{\Lambda}(t)(1 - \exp\{-\dt \bar{\Lambda}(t)\})].%
\end{eqnarray}
As in Theorem~\ref{thm:suf-cond-equidisp}, it follows by dominated convergence, since $\bar{z}\bar{\lambda}(1 - \exp\{-\dt \bar{\lambda}\} \leq \bar{z}\bar{\lambda})$ and $E[\bar{Z}(t) \bar{\Lambda}(t)]$ is finite, that
\begin{eqnarray}
\nonumber \lim\limits_{\dt \downarrow 0} \frac{E[ \bar{Z}(t) \dt \bar{\Lambda}(t)(1 - \exp\{-\dt \bar{\Lambda}(t)\})]}{\dt} &=& E[\lim\limits_{\dt \downarrow 0} \bar{Z}(t) \bar{\Lambda}(t)(1 - \exp\{-\dt \bar{\Lambda}(t)\})] = 0.%
\end{eqnarray}
Therefore, $E[\Delta N_{ij}(t) \Ind{S^c}]=o(\dt)$ and the result for the mean follows.
Replacing first by second moments, the result for the variance follows since
\begin{eqnarray}
\nonumber E[(\Delta N_{ij}(t))^2 \Ind{S^c}] &\leq& E[(\Delta \bar{N}_{ij}(t))^2 \Ind{S^c}]\\
\nonumber &=& E[E[(\Delta \bar{N}_{ij}(t))^2 \Ind{S^c}|\bar{\Lambda}(t),\bar{Z}(t)]]\\
\nonumber &=& E[\bar{Z}^2(t) \dt \bar{\Lambda}(t) + \bar{Z}^2(t) \dt^2 \bar{\Lambda}^2(t) - \bar{Z}^2(t)\dt \bar{\Lambda}(t) \exp{-\dt \bar{\Lambda}(t)}]\\
\nonumber &\leq& E[2 \bar{Z}^2(t) \dt^2 \bar{\Lambda}^2(t)]=o(\dt),%
\end{eqnarray}
since $E[\bar{Z}^2(t) \bar{\Lambda}^2(t)]$ is assumed to be finite.
\end{proof}

Theorem~\ref{thm:suf-cond-mult-build-block} makes it straightforward to characterize dispersion of MCSs according to whether the $ij$-marginal processes associated to $\{\bm{X}(t)\}$ are simple or compound.
\begin{corollary}[sufficient and necessary conditions for Markov infinitesimal dispersion]\label{cor:suf-cond-overdisp}
Let $\{\bm{X}(t)\}$ be a time homogeneous, stable and conservative Markov counting system with associated multivariate counting process $\{\bm{N}(t)\}$.
Supposing (P3) and (P4), it is sufficient and necessary that an $ij$-marginal processes $\{N_{ij}(t)\}$ associated with $\{\bm{X}(t)\}$ be compound (simple) for it to be infinitesimally over-(equi-)dispersed.
\end{corollary}
\begin{proof}
Let us first establish sufficiency.
Using Theorem~\ref{thm:suf-cond-mult-build-block}, for those $ij$-marginal processes $\{N_{ij}(t)\}$ which are simple,
\bean
D^{ij}_{d\bm{X}}(\bm{x}) &=& \frac{\sum_k k^2 \q[\bm{x},k][ij][]}{\sum_k k\q[\bm{x},k][ij][]} = \frac{\q[\bm{x},1][ij][]}{\q[\bm{x},1][ij][]} = 1.
\eean
For those $ij$-marginal processes $\{N_{ij}(t)\}$ which are compound,
\bean
D^{ij}_{d\bm{X}}(\bm{x}) &=& \frac{\sum_k k^2 \q[\bm{x},k][ij][]}{\sum_k k\q[\bm{x},k][ij][]} > 1
\eean
since $k^2 > k$ for $k \;\inNatII$.
Hence, an infinitesimally equi-dispersed process must be simple because compoundness suffices to establish over-dispersion.
Analogously, an infinitesimally over-dispersed process must be compound, establishing necessity.
\end{proof}
As an illustration of this result, consider the following MCS.
Define the sets $\MultI^{BD}=\{\textit{``unborn''},\textit{``alive''}\}$ and $\MultII^{BD} = \{\textit{``alive''},\textit{``dead''}\}$  so that $\MultSet^{BD} = \{\textit{``unborn'', }\textit{``alive'', }\textit{``dead''}\}$ which we relabel as $\MultSet^{BD} = \{1,2,3\}$. Then, let $\{\bm{X}^{BD}(t)\}$ with starting conditions $\big(X_1^{BD}(0),X_2^{BD}(0),X_3^{BD}(0)\big) \in \Z\times\N\times\N$ be the MCS with associate multivariate counting process $\{\bm{N}^{BD}(t)\} = \big\{ N_{12}^{BD}(t),\; N_{21}^{BD}(t),\; N_{13}^{BD}(t),\; N_{31}^{BD}(t),\; N_{23}^{BD}(t),\; N_{32}^{BD}(t)\big\}$ with $N_{ij}^{BD}(t)=0$ for $ij \in \big\{21,13,31,32\big\}$ and defined by infinitesimal probabilities for $k >0$
\bean
\q[\bm{x},k][12][BD] &=& \beta x_2 \Ind{x_2>0, k=1}\\
\q[\bm{x},k][23][BD] &=& \delta x_2 \Ind{x_2>0, k=1}.
\eean
$\{X_2(t)\}$ is commonly referred to as simple, linear birth-death process with death rate $\delta$ and birth rate $\beta$.
This MCS representation of the process departs from its usual univariate representation by extending the state to include $X_1(t) = -N_{12}(t)$, which keeps track of the (negative) number of total births, and $X_3(t) = N_{23}(t)$, which keeps track of the number of total deaths.

\begin{corollary}[infinitesimal equi-dispersion of birth-death processes]\label{cor:equi-disp-birth-death}
A simple, linear birth-death process with death rate $\delta \in R^+$ and birth rate $\beta \in R^+$ is infinitesimally equi-dispersed.
\end{corollary}

\begin{proof} Use Corollary~\ref{cor:suf-cond-overdisp}.
Since the maximum jump of both non-zero $ij$-marginal processes $\{N_{12}^{BD}(t)\}$ and $\{N_{23}^{BD}(t)\}$ is one, we only need to check existence of $E\big[\bar{\Lambda}_{ij}^{BD}(t)| \bm{X}^{BD}(t) = \bm{x} \big]$ and of $V\big[\bar{\Lambda}_{ij}^{BD}(t)| \bm{X}^{BD}(t) = \bm{x} \big]$.
This is granted since
\[
\lambda^{BD}(\bm{X}^{BD}(t)) = (\beta + \delta) X_2^{BD}(t) \Ind{X_2^{BD}(t)>0}
\]
is stochastically bounded by $(\beta + \delta)B(t)$, where $\{B(t)\}$ is a simple, linear birth process with birth rate $\beta$.
Then
$$E\big[ \bar \Lambda_{ij}^{BD}(t) | \bm{X}^{BD}(t) = \bm{x} \big] \le E\big[\sup_{t \le s \le t + \bar\dt} (\beta + \delta)B(s)\; | B(t) = x_2 \big] = E\big[(\beta + \delta)B(t + \bar\dt)\; | B(t) = x_2 \big] < \infty.$$
By the same argument $V\big[\bar \Lambda_{ij}^{BD}(t) | \bm{X}^{BD}(t) = \bm{x}\big] < \infty$ and the result follows for the non-degenerate case of $x_2 >0$.
\end{proof}

\section{Over-dispersed Markov Counting Systems}\label{sec:OMCS}
After establishing equi-dispersion of simple multivariate processes in Corollary~\ref{cor:suf-cond-overdisp}, we proceed to construct over-dispersed multivariate processes.
As we have advanced, we will pursue this goal using a ``building block'' strategy.
Moving from a collection of univariate blocks to a multivariate MCS is not immediate.
In our implementation of the ``building block'' strategy, we proceed in two steps.
First, we illustrate how a univariate MCP may be equivalently written as a bivariate MCS by endowing it with sets $\MultI$ and $\MultII$ and adding an additional state.
We do this to facilitate notation in the rest of the paper.
Second, we show in Lemma~\ref{lem:disp-MCS-ass-block} that ``stacking'' the infinitesimal probabilities of each block (now written as the corresponding bivariate MCS) results in a MCS that retains dispersion of the blocks.


Consider writing the MCP $\{B(t)\}$, the simple, linear birth process of Corollary~\ref{cor:equi-birth}, with starting value $B(0) \in \N$, as bivariate MCS $\{\bm{X}^B(t)\}$.
This may be achieved by defining the sets $\MultI^B=\{\textit{``unborn''}\}$ and $\MultII^B = \{\textit{``alive''}\}$, so that $\MultSet^B = \{\textit{``unborn''},\textit{``alive''}\}$ which we relabel as $\MultSet^B = \{1,2\}$.
Then, let $\{\bm{X}^B(t)\}$ be the bivariate MCS with associated multivariate counting process $\{\bm{N}^B(t)\} = \Big\{ N_{12}^B(t),\; N_{21}^B(t)\Big\}$ and with $X_1^B(0) \in \Z$ and $X_2^B(0) \equiv B(0)$,
\bean
\q[\bm{x},k][12][B] &=& \beta x_2 \Ind{x_2 > 0}\\
\q[\bm{x},k][21][B] &=& 0.
\eean
Note that, by the definition of the infinitesimal probabilities, it follows that
$N^B_{21}(t) = 0$ and that $B(t)-B(0) \sim N_{12}^B(t)$.
Had $\{B(t)\}$ been the negative binomial gamma process of Section~\ref{sec:negbin-gamma}, the corresponding MCS could have been defined letting
\bean
\q[\bm{x},k][12][B] &=& {x_2 + k - 1 \choose k} \sum\limits_{j=0}^{k}{{k \choose j}}
(-1)^{k - j + 1}(\tau^B)^{-1} \ln{\bigl(1  + \beta \tau^B (x_2 + k - j)\bigr)}\\
\q[\bm{x},k][21][B] &=& 0.
\eean

Consider now writing the MCP $\{D(t)\}$, the counting process associated with a simple, linear death process $\{\tilde D(t)\}$ of Corollary~\ref{cor:equi-death}, with starting value $\tilde D(0) \in \N$, as the bivariate MCS $\{\bm{X}^D(t)\}$.
This may be achieved by defining now sets $\MultI^D=\{\textit{``alive''}\}$ and $\MultII^D = \{\textit{``dead''}\}$, so that $\MultSet^D = \{\textit{``alive''},\textit{``dead''}\}$ which we relabel as $\MultSet^D = \{1,2\}$.
Then, letting $\{\bm{X}^D(t)\}$, with $X_1^D(0) \equiv \tilde D(0)$ and $X_2^D(0) \in \N$, be the bivariate MCS with associated multivariate counting process $\{\bm{N}^D(t)\} = \Big\{ N_{12}^D(t),\; N_{21}^D(t)\Big\}$ defined by infinitesimal probabilities for $k >0$
\bean
\q[\bm{x},k][12][D] &\equiv& \delta x_1 \Ind{x_1 > 0}\\
\q[\bm{x},k][21][D] &\equiv& 0.
\eean
Again $N_{21}(t) = 0$ and $D(t) \sim N_{12}^D(t)$ and the following alternative infinitesimal probabilities would define the bivariate MCS corresponding to the binomial gamma process of Section~\ref{sec:binom-gamma}
\bean
\q[\bm{x},k][12][D] &\equiv& {x_1 \choose k} \sum\limits_{j=0}^{k}{{k \choose j}}(-1)^{k - j + 1}(\tau^D)^{-1} \ln{\bigl(1  + \delta \tau^D (x_1 - j)\bigr)}\\
\q[\bm{x},k][21][D] &\equiv& 0.
\eean

Now that we have a potential collection of blocks (bivariate MCSs corresponding to univariate MCPs), we move to the second step of our strategy and define in Definition~\ref{def:MCS-ass-blocks} a MCS associated to a collection of blocks by ``stacking'' the infinitesimal probabilities of these blocks and show in Lemma~\ref{lem:disp-MCS-ass-block} that the block dispersion remains unchanged.
Note that Definition~\ref{def:MCS-ass-blocks} complements Definition~\ref{eqn:MCS} in which MCSs were defined.
From Definition~\ref{eqn:MCS}, it is possible to decompose a MCS into a collection independent univariate MCPs of the $ij$-type by considering only sets of $k$ infinitesimal probabilities $\q[\bm{x},k][ij][]$ and their corresponding initial conditions.
On the other hand, Definition~\ref{def:MCS-ass-blocks} gives one possible construction of a MCS given a collection of independent univariate MCPs.
We illustrate these results by constructing an over-dispersed birth-death process in Proposition~\ref{pro:overdisp-birth-death}.
The general class of MCSs specified in Definition~\ref{def:MCS-ass-blocks} also includes, for example, the over-dispersed compartment model of \citep{breto2009-aoas}.

%

\begin{definition}[MCS associated to a collection of
blocks]\label{def:MCS-ass-blocks}
Let $\Big\{ \{\bm{X}^{\BlockInd}(t)\} : \BlockInd \in \BlockSet \Big\}$ be a collection of independent time homogeneous, conservative and stable bivariate MCSs defined on sets $\MultI^{\BlockInd} \equiv \{i^{\BlockInd}\}$ and $\MultII^{\BlockInd} \equiv \{j^{\BlockInd}\}$ by infinitesimal probabilities 
$\q[\bm{x},k][][\BlockInd]$.
We define the MCS associated with this collection of blocks to be the MCS $\{\bm{W}(t)\}$ defined on sets $\MultI \equiv \bigcup_{\BlockInd \in \BlockSet} \MultI^{\BlockInd}$ and $\MultII \equiv \bigcup_{\BlockInd \in \BlockSet} \MultII^{\BlockInd}$ via infinitesimal probabilities $\q[\bm{w},k][ij][] \equiv \q[(w_i,w_j),k][][B(i,j)]$ where $B: \MultI \times \MultII \rightarrow \BlockSet$
is defined so that $B(i,j)=\BlockInd$ implies $\MultI^{\BlockInd} = i$ and $\MultII^{\BlockInd} = j$.
\end{definition}

For Definition~\ref{def:MCS-ass-blocks} to uniquely specify a MCS, we need to impose two conditions.
First, since $B(i,j)$ is the element in the block set $\BlockSet$ corresponding to transitions from $i$ to $j$, we require that only one block in the collection counts events of the $ij$-type, so that $B(i,j)$ defines only one element in the block set $\BlockSet$,  i.e. so that $\MultI^{\BlockInd} = i$ and $\MultII^{\BlockInd} = j$ together imply $B(i,j)=\BlockInd$.
Then, the assignment of infinitesimal probabilities in Definition~\ref{def:MCS-ass-blocks} is unique.
Second, some care is needed with initial conditions of the blocks and how they correspond to $\bm{W}(0)$, the initial conditions of the MCS.
We simply assume that $\bm{W}(0)$ is defined separately from $\Big\{ \{\bm{X}^{\BlockInd}(0)\} : \BlockInd \in \BlockSet \Big\}$, to avoid concerning ourselves with possible inconsistencies between the inital conditions of the blocks once they are juxtaposed.

\begin{lemma}[infinitesimal dispersion of MCSs associated with a collection of blocks]\label{lem:disp-MCS-ass-block}
Let $\{\bm{W}(t)\}$ be the MCS associated with collection of blocks $\Big\{ \{\bm{X}^{\BlockInd}(t)\} : \BlockInd \in \BlockSet \Big\}$.
Then, supposing (P3) and (P4), the $ij$-marginal processes of $\{\bm{W}(t)\}$ have the same dispersion as the $ij$-marginal processes of $\Big\{ \{\bm{X}^{\BlockInd}(t)\}: \BlockInd \in \BlockSet \Big\}$.
\end{lemma}
\begin{proof}
Using Theorem~\ref{thm:suf-cond-mult-build-block} and letting $\bm{\tilde X}(t) \equiv \bm{X}^{B(i,j)}(t)$ to simplify subindices,
\bean
D_{d\bm{W}}^{ij}(\bm{w}) &=& \frac{\sum k^2 \q[\bm{w},k][ij][]}{\sum k \q[\bm{w},k][ij][]}\\
D_{d\bm{\tilde X}}^{ij}(w_i,w_j) &=& \frac{\sum k^2 \q[(w_i,w_j),k][ij][B(i,j)]}{\sum k \q[(w_i,w_j),k][ij][B(i,j)]},
\eean
and it follows, by Definition~\ref{def:MCS-ass-blocks}, that
\[
D_{d\bm{W}}^{ij}(\bm{w}) = D_{d\bm{\tilde X}}^{ij}(w_i,w_j)
\]
\end{proof}

In light of the equi-dispersion shown in Corollary~\ref{cor:equi-disp-birth-death}, consider applying Lemma~\ref{lem:disp-MCS-ass-block} to construct an over-dispersed birth-death process based on the over-dispersed processes of Section~\ref{sec:OMCP}.

\begin{proposition}[infinitesimally over-dispersed birth-death process]\label{pro:overdisp-birth-death}
Consider the collection of blocks $\Big\{\{\bm{X}^B(t)\}, \{\bm{X}^D(t)\}\Big\}$.
Here, $\{\bm{X}^B(t)\}$ is a bivariate MCS corresponding to the infinitesimally over-dispersed negative binomial gamma process of Section~\ref{sec:negbin-gamma} and $\{\bm{X}^D(t)\}$ is a bivariate MCS corresponding to the infinitesimally over-dispersed binomial gamma process of Section~\ref{sec:binom-gamma}.
Consider $\{\bm{W}(t)\}$, the MCS associated to this collection defined, as in Definition~\ref{def:MCS-ass-blocks}, on sets $\MultI \equiv \{\textit{``unborn''},\textit{``alive''}\}$ and $\MultII \equiv \{\textit{``alive''},\textit{``dead''}\}$ and $\MultSet = \{\textit{``unborn''},\textit{``alive''},\textit{``dead''}\}$ ,relabeled as $\MultSet = \{1,2,3\}$, with initial conditions $\bm{W}(0)$
, via infinitesimal probabilities for $k\; \inNatII$
\begin{eqnarray*}
\nonumber \q[\bm{w},k][12][] &=& {w_2 + k - 1 \choose k} \sum\limits_{j=0}^{k}{{k \choose j}}
(-1)^{k - j + 1}(\tau^B)^{-1} \ln{\bigl(1  + \beta \tau^B (w_2 + k - j)\bigr)}
\end{eqnarray*}
and, for $k \in \{0,\dots,w_1\}$ and $w_2 >0$,
\begin{eqnarray}
\nonumber \q[\bm{w},k][23][] &=& {w_2 \choose k} \sum\limits_{j=0}^{k}{{k \choose j}}
(-1)^{k - j + 1}(\tau^D)^{-1} \ln{\bigl(1  + \delta \tau^D (w_2 - j)\bigr)}
\end{eqnarray}
and zero otherwise.
The infinitesimal dispersion of $\{\bm{W}(t)\}$ is
\bean
D_{d\bm{W}}^{12}(\bm{w}) &=& 1 + (w_2 - 1)\Biggl[\frac{2\ln(1-\beta\tau^B)- \ln(1-2\beta\tau^B)}{-\ln(1-\beta\tau^B)}\Biggr]\\
D_{d\bm{W}}^{23}(\bm{w}) &=& 1 + (w_2 - 1)\Biggl[\frac{2\ln(1+\delta\tau^D)- \ln(1+2\delta\tau^D)}{\ln(1+\delta\tau^D)}\Biggr].
\eean
for $2\beta \tau^B < 1$.
Hence, the $ij$-marginals of $\{\bm{W}(t)\}$ are infinitesimally over-dispersed for $w_2 > 1$ and equi-dispersed for $w_2 = 1$.
\end{proposition}
\begin{proof}
This proposition follows analogously to Corollary~\ref{cor:equi-disp-birth-death} using Lemma~\ref{lem:disp-MCS-ass-block}.
In this case, the rate function is
\bean
\lambda(\bm{W}(t)) &=& \sum\limits_k \q[\bm{W}(t),k][12][] + \sum\limits_k \q[\bm{W}(t),k][23][]\\
                   &=& (\tau^B)^{-1} \ln\big(1 + \beta \tau^B W_2(t)\big) + (\tau^D)^{-1} \ln\big(1 + \delta \tau^D W_2(t)\big)\\
                   & \leq & (\beta + \delta) W_2(t).
\eean
$W_2(t)$ is now stochastically bounded by $B(t)$, where $\{B(t)\}$ a negative binomial gamma process playing the same role as the simple, linear birth process in the proof of  Corollary~\ref{cor:equi-disp-birth-death}.
Hence, $\bar \Lambda_{ij}(t)$ is conditionally stochastically bounded by $(\beta + \delta)B(t + \dtt)$.
Another difference with Corollary~\ref{cor:equi-disp-birth-death} is that both $\{N_{12}(t)\}$ and $\{N_{23}(t)\}$ may jump by more than one unit, so we also need to stochastically bound $\bar Z_{ij}$, the maximum jump in $[t, t + \dtt]$ of the marginals.
Note that the maximum jump of the birth marginal $\{N_{12}(t)\}$ may be bound by $B(t + \dtt)$, i.e. any jump can be at most all of those born in that time interval.
The maximum jump of the death marginal $\{N_{23}(t)\}$ may be bound by $\{W_2(t) + B(t)\}$, i.e. any jump can be at most all those alive at time $t$ plus all those born in that time interval.
Then
\bean
E\big[\bar Z_{12}(t) \bar \Lambda_{12}(t) | \bm{W}(t) = \bm{w} \big] & \leq & E\big[(\beta + \delta)B^2(t + \dtt) | B(t) = w_2 \big] < \infty\\
E\big[\bar Z_{23}(t) \bar \Lambda_{23}(t) | \bm{W}(t) = \bm{w} \big] & \leq & E\big[(\beta + \delta)B(t + \dtt) w_2 + B^2(t + \dtt)] | B(t) = w_2 \big] < \infty
\eean
Hence, by finiteness of the fourth moment of the negative binomial gamma process, (P3) and (P4) hold and the result follows.
\end{proof}

Note that replacing $\q[\bm{w},k][12][]$ or $\q[\bm{w},k][23][]$ by the corresponding equi-dispersed infinitesimal probabilities in Proposition~\ref{pro:overdisp-birth-death}is possible and would yield infinitesimal over-dispersion in only one of the $ij$-marginal processes.

\section{Discussion}\label{sec:disc}

We have shown in Section~\ref{sec:notation} that simultaneous events are required in order to obtain infinitesimal over-dispersion in MCPs.
We now discuss heuristic interpretations and applications of this result.
There are two distinct motivations for modeling simultaneous events: the process in question may indeed have such occurrences, or the process may have clusters of event times that are short compared to the scale of primary interest.
In many applications, only aggregated counts and not event times may be available, in which case any clustering time scale which is shorter than the aggregation timescale may be appropriately modeled by simultaneous events.
Either way, the conclusion remains that if one wishes to use models based on Markov counting processes which match infinitesimal dispersion characteristics of a system then the possibility of simultaneous events in these models is unavoidable.

In the past, a modeling hypothesis that events occur non-simultaneously seems to have been favored.
As already proved, this comes at odds with the desire of a better fit to data by allowing Markov processes to generate additional variability.
To help reconcile intuition built on simple point process models with the models introduced in this paper, we give an interpretation of infinitesimally over-dispersed models based on the idea that the additional variability comes in the form of clusters of events which infinitesimally turn into simultaneous events as the time length of the cluster tends to zero.

Consider the bivariate process $(N(t),M(t))$ where $\{N(t)\}$ is a simple univariate MCP with conditional infinitesimal event probability $\q[n,1][][] = \lambda(n)M(t)$ and $M(t)$ is a discrete-time noise process which is constant over time intervals $\{[t_i,t_i+\omega]:t_0, \omega \in \R^+, t_i = t_{i-1}+\omega, i \;\inNatII\}$ and has $E[M(t)]=1$.
If the values of $M(t)$ on distinct intervals $[t_i,t_{i+1}]$ were independently gamma distributed with variance inversely proportional to $\omega$ and we let $\omega \downarrow 0$, we heuristically would obtain the gamma processes from Section~\ref{sec:OMCP}.

Consider a time interval $[t, t+ \dt]$ with $\dt >> \omega$. In this new model (before letting $\omega \downarrow 0$), the single events would form clusters over time since there would be more events in the $\omega$-intervals where the integrated event rate was higher than the mean $\lambda(n) \omega$ and fewer otherwise. More extreme (sudden) variations in the event rates would produce stronger clustering. This dependence induced by the clustering in turn increases the heterogeneity for a given mean infinitesimal rate $\lambda(n)$.
This interpretation is parallel to the common practice of modeling binomial over-dispersed experiments letting the parameter $p$ in a set of binomial experiments be stochastic, where the additional variability reflects the dependence between the $N$ individuals in each experiment.

Now let us consider more carefully the limit of this process.
Still considering time interval $[t, t + \dt]$, clusters start happening over shorter time intervals as $\omega \downarrow 0$ and there will be fewer and fewer events in each $\omega$-interval which will tend to obscure the clustering.
However, clusters may still be perceived if there are rare realizations of $M(t)$ which are extremely different from values in nearby $\omega$-intervals even as $\omega$ becomes small.
One possibility to make sure these extreme differences for which clustering will still be apparent is to ensure there is enough variability in $M(t)$ as $\omega$ decreases.
It turns out that by letting $M(t)$ be integrated white noise, one gets $V[M(t)] \propto \omega^{-1} \to \infty$ as $\omega \to 0$. This guarantees that $V[M(t)]$ is large enough for clustering to be present in $[t, t+\dt]$ as long as $\omega$ decreases as fast or faster than $\dt$, even in the limit as $\dt \downarrow 0$, i.e. infinitesimally. This heuristic leads to seeing simultaneous events as clusters of single events in intervals of length zero, i.e. simultaneous events can be seen as the limit of single-event clusters.

Based on this interpretation of simultaneous events, two questions might be addressed.
First, infinitesimally over-dispersed MCPs might be useful in applications where, even though exactly simultaneous events may be considered impossible, single events can be clustered very tightly compared to the distance between event clusters.
Then, infinitesimally over-dispersed MCPs may be a useful Markovian approximation to more complex, non-Markovian simple processes.
In this approximation these tight clusters are approximated by clusters with distance between cluster members equal to zero. One instance of applications that fall in this category is infectious diseases where it is hard to imagine ever having event times data and where it is plausible that a group of susceptible get infected with very short inter-event times compared to the time until another infectious individual infects another susceptible.

Second, if the goal is to model a process where it is natural to consider simultaneous events, there are other mechanisms that could be used, like a build-up of single events which, passed a threshold, becomes a batch event. Applications falling in this category include production systems and rental businesses \cite{ormeci2005}, physical processes and quantum optics \cite{gillespie2005} and internet traffic \cite{klemm2003}.
In this case, the infinitesimally over-dispersed MCPs presented here might serve as a null ``Markovian'' model against which to test other more intricate models for dependence.

\section*{Acknowledgements}

This work was supported by National Science Foundation grant DMS-0805533, the RAPIDD program of the Science \& Technology Directorate, Department of Homeland Security, the Fogarty International Center, National Institutes of Health and Spanish Government Projects SEJ2006-03919 and ECO2009-08100. This work was conducted at the Inference for Mechanistic Models Working Group supported by the National Center for Ecological Analysis and Synthesis, a Center funded by NSF (Grant \#DEB-0553768), the University of California, Santa Barbara, and the State of California.


\appendix


\section{Proofs for univariate Markov counting processes}\label{app:proof-overd-univ}

\begin{proof}[\textbf{Proof of Theorem~\ref{thm:suf-cond-mixed-equidisp}}(sufficient condition for mixed Markov infinitesimal equi-dispersion)]
Letting the process be a mixed process means that $\bar{\Lambda}(t)$ is now stochastic because of its dependence on the process $\{N(t)\}$, as in the non-mixing case, but also on the random variable $M$.
The result for the mean follows again by dominated convergence but the dominated functions are now
\begin{eqnarray*}
\int \lambda(1-\exp\{-\dt \lambda\}) f_{\bar{\Lambda}|M,N(t)=n}(\lambda, m) \, d\lambda \leq  \int \lambda f_{\bar{\Lambda}|M,N(t)=n}(\lambda, m) \, d\lambda,
\end{eqnarray*}
for all $m$, and the dominating function has a finite integral since $E[\bar{\Lambda}(t)]$ is assumed to be finite, i.e.
\begin{eqnarray*}
\int \int \lambda f_{\bar{\Lambda}|M,N(t)=n}(\lambda, m) \, d\lambda \,\,f_{M|N(t)=n}(m) \, dm = E[E[\bar{\Lambda}(t)|M]] < \infty.
\end{eqnarray*}
Then
\begin{eqnarray*}
\lim\limits_{\dt \downarrow 0} \frac{E[E[\dt \bar{\Lambda}(t)(1 - \exp\{-\dt \bar{\Lambda}(t)\})|M]]}{\dt} &=& E[\lim\limits_{\dt \downarrow 0} \bar{\Lambda}(t,M)(1 - \exp\{-\dt \bar{\Lambda}(t,M)\})] = 0.%
\end{eqnarray*}
The result for the variance follows again in the same lines as for the non-mixing case, i.e.
\begin{eqnarray*}
E[(\Delta N(t))^2 \Ind{\Delta N(t) > 1}] &=& E[E[(\Delta N(t))^2 \Ind{\Delta N(t) > 1}|M]]\\
&\leq& 2\dt^2 E[E[\bar{\Lambda}^2(t)|M]] = o(\dt),
\end{eqnarray*}
by the assumption that $E[\bar{\Lambda}^2(t)] < \infty$.
These two results show that the same terms that vanished in Theorem~\ref{thm:suf-cond-equidisp} vanish now as well.
Then,
\begin{eqnarray}
\nonumber E[\Delta N(t)] = E[\Ind{\Delta N(t)=1}] + o(\dt) &=& E[E[\Ind{\Delta N(t)=1}|M]] + o(\dt)\\
\label{eqn:cor-mixed-MCP} &=& E[\Lambda(n) \exp\{-\dt \Lambda(n+1)\} \phi(\dt)] + o(\dt)\\
\nonumber &=& E[\dt \Lambda(n) + o(\dt)] + o(\dt)
\end{eqnarray}
where~\eqref{eqn:cor-mixed-MCP} follows by Lemma~\ref{lem:prob-Scomp-little-oh}.
Here
\[
\phi(\dt) = \left\{
\begin{array}{l l}
  \dt & \quad \text{if $\Lambda(n)=\Lambda(n+1)$}\\
  {\displaystyle \frac{1 - \exp\{- \dt (\Lambda(n) - \Lambda(n+1))\}}{\Lambda(n) - \Lambda(n+1)}} & \quad \text{if $\Lambda(n) \neq \Lambda(n+1)$}\\
\end{array} \right.
\]
For $\Lambda(n) > \Lambda(n+1)$, $\phi(\dt) \leq h$ and for $\Lambda(n) < \Lambda(n+1)$, $\phi(\dt) = O(\dt)$.
For $\dt$ small enough, $\phi(\dt) \leq 1$ and the functions inside the expected value in~\eqref{eqn:cor-mixed-MCP}
are bounded by $\Lambda(n)$.
Then, taking limits gives the desired result via dominated convergence
\bean
\lim\limits_{\dt \downarrow 0}\frac{E[\Delta N(t)]}{h} &=& E[\Lambda(n)],
\eean
where $E[\Lambda(n)] \le E[\bar \Lambda(t)] < \infty$.
The same argument gives $\lim\limits_{\dt \downarrow 0}\dt^{-1}E[(\Delta N(t))^2] = E[\Lambda(n)]$ and the dispersion results of Theorem~\ref{thm:suf-cond-equidisp} follow.
\end{proof}

\begin{proof}[\textbf{Proof of Proposition~\ref{pro:Pois-gamma-proc}} (Poisson gamma process)]
We denote the Poisson rate $\rho$ in this proof and reserve $\alpha$ for the gamma shape parameter.
Since $\{N(t)\}$ is a conditional Poisson process,
\bea
\nonumber P(\Delta N(t) = k|N(t),\Delta\SIP(t)) &=& \frac{e^{-\rho \Delta\SIP(t)}(\rho \Delta\SIP(t))^k}{k!}.
\eea
It is a standard result that if $\rho \Delta \SIP(t)$ follows a gamma distribution with mean $\rho \dt$ and variance $\rho^2 \tau \dt$
the distribution of the increments of $\{N(t)\}$ is negative binomial with probability mass function
\begin{eqnarray}
\label{eqn:Pois-gamma-pmf}P(\Delta N(t)=k|N(t)=n) &=& \frac{G\left(\tau^{-1}\dt + k\right)}{k!G(\tau^{-1} \dt)}p^{\tau^{-1}\dt}\left(1-p\right)^k.
\end{eqnarray}
with $p=(1+\tau\rho)^{-1}$. The limiting probabilities follow by a Taylor series expansion about $\dt=0$:
\begin{eqnarray*}
\nonumber P(\Delta N(t)=0|N(t)=n) &=& p^{\tau^{-1} \dt} = 1 + \tau^{-1} \log{(p)} \dt + o(\dt)\\%
\nonumber P(\Delta N(t)=k|N(t)=n) &=& \frac{\Bigl(\tau^{-1} \dt + o(\dt)\Bigr)}{k}\times \Bigl(1 + \tau^{-1} \log{(p)} \dt + o(\dt)\Bigr)\left(1-p\right)^k\\
\nonumber &=& \frac{\tau^{-1} \left(1-p\right)^k}{k}\dt + o(\dt),
\end{eqnarray*}
for $k > 0$, using in~\eqref{eqn:Pois-gamma-pmf}
\begin{eqnarray*}
\frac{G(\eta + k)}{k!G(\eta)} &=& k!^{-1}(\eta + k-1) \times (\eta + k-2) \times \ldots \times (\eta + 2) \times (\eta + 1) \times (\eta)\\
&=& \sum\limits_{j=0}^{k-1}k!^{-1}\phi_j\dt^j \times \eta = k^{-1}\eta + \sum\limits_{j=1}^{k-1}k!^{-1}\phi_j\dt^{j+1}\tau^{-1}\\
&=& k^{-1}\eta + o(\dt),%
\end{eqnarray*}
with $\eta = \tau^{-1} \dt$. Recalling that $p=(1+\tau\alpha)^{-1}$, the moments follow by
\bean
E[\Delta N(t)|N(t)=n] &=& \frac{\tau^{-1}\dt(1-p)}{p}\\
&=& (1+ \tau \alpha)\tau^{-1}\dt - \tau^{-1}\dt = \alpha \dt\\
V[\Delta N(t)|N(t)=n] &=& \frac{\tau^{-1}\dt(1-p)}{p^2} = (1+ \tau \alpha)\alpha\dt
\eean
\end{proof}

\begin{proof}[\textbf{Proof of Proposition~\ref{pro:binom-gamma-proc}} (binomial gamma process)] Since $\{N(t)\}$ is the counting process associated with a conditional linear death process, the increment process is binomial with parameters size $\tilde{n}$ and event probability $\Pi(t) = 1-e^{-\delta \Delta \SIP(t)}$, i.e.
\begin{eqnarray*} P(\Delta N(t) = k|N(t)=n,\Delta \SIP(t)) &=& {\tilde{n} \choose k} \Pi(t)^{k}(1 - \Pi(t))^{\tilde{n}-k},
\end{eqnarray*}
for $k \in \{0,1,\dots,\tilde{n}\}$.
We integrate out the continuous-time gamma noise using the fact that
$\delta \Delta \SIP(t)$ follows a gamma distribution with mean $\delta \dt$ and variance $\delta^2 \tau \dt$
and completing the resulting incomplete gamma density making use of the multinomial theorem as follows
\begin{align*}
&P(\Delta N=k| N(t)=n) = \int\limits_{0}^{\infty}
    {
    {\tilde{n} \choose k}
    \bigl[1 - e^{-x}\bigr]^{k}
    \bigl[e^{-x}\bigr]^{\tilde{n} - k}
    \frac{x^{\alpha-1}e^{-x\beta}\beta^{\alpha}}{G(\alpha)}
    }dx\\
&\quad =
    {\tilde{n} \choose k}
    \int\limits_{0}^{\infty}{
    \Biggl[\sum\limits_{j=0}^{k}{
    {k \choose j}
    }(- e^{-x})^{k-j}\Biggr]e^{-x(\tilde{n} - k)}
    \frac{x^{\alpha-1}e^{-x\beta}\beta^{\alpha}}{\Gamma(\alpha)}
    }dx\\
&\quad =
    {\tilde{n} \choose k}
    \int\limits_{0}^{\infty}{\sum\limits_{j=0}^{k}{
    {k \choose j}
    }(-1)^{k-j}e^{-x(\tilde{n} - j)}
    \frac{x^{\alpha-1}e^{-x\beta}\beta^{\alpha}}{\Gamma(\alpha)}
    }dx\\
&\quad =
    {\tilde{n} \choose k}
   \sum\limits_{j=0}^{k}{
    {k \choose j}
    }(-1)^{k-j}\frac{\beta^{\alpha}}{(\beta  + \tilde{n} - j)^{\alpha}}
\times \int\limits_{0}^{\infty}{\frac{x^{\alpha-1}e^{-x(\beta + \tilde{n} - j)}(\beta  + \tilde{n} - j)^{\alpha}}{\Gamma(\alpha)}}dx\\
&\quad =
    {\tilde{n} \choose k}
    \sum\limits_{j=0}^{k}{
    {k \choose j}
    }(-1)^{k-j}
    \Bigl(1  + \delta \tau(\tilde{n} - j)\Bigr)^{-\dt\tau^{-1}}
\end{align*}
for $k \in \{0,\dots,\tilde{n}\}$ and recalling that $\alpha=\dt \tau^{-1}$ and $\beta=\delta^{-1}\tau^{-1}$.
The limiting probabilities follow by a Taylor series expansion about $\dt=0$:
\begin{eqnarray*}
P(\Delta N(t)=0|N(t)=n) &=&\bigl(1  + \delta \tau \tilde{n}\bigr)^{-\dt\tau^{-1}} = 1 - \tau^{-1}\ln{\bigl(1  + \delta \tau \tilde{n}\bigr)}\dt + o(\dt)\\%
P(\Delta N(t)=k|N(t)=n) &=&
    {\tilde{n} \choose k}
    \sum\limits_{j=0}^{k}{
    {k \choose j}
    }(-1)^{k - j}
    \Biggl(1 -\tau^{-1}\ln{\bigl(1  + \delta \tau (\tilde{n} - j)\bigr)}\dt + o(\dt)\Biggr)\\
&=&{\tilde{n} \choose k}
   \sum\limits_{j=0}^{k}{
    {k \choose j}
    }(-1)^{k - j + 1}\tau^{-1}\ln{\bigl(1  + \delta \tau (\tilde{n} - j)\bigr)}\dt + o(\dt),
\end{eqnarray*}
for $k \geq 1$, since by the binomial theorem $\sum\limits_{j=0}^{k}{{k \choose j}}(-1)^{k - j}=(1-1)^k=0$.

The moments are
\begin{eqnarray}
 \nonumber E[\Delta N(t)|N(t)=n]&=& \tilde{n}E[\Pi(t)|N(t)=n] = \tilde{n}E[1-e^{-\delta \Delta \SIP(t)}|N(t)=n]\\
 \nonumber V[\Delta N(t)|N(t)=n]&=& V[\tilde{n}\Pi(t)|N(t)=n] + E[\tilde{n}\Pi(t)(1-\Pi(t))|N(t)=n]\\
 \nonumber                      &=& E[\Delta N(t)|N(t)=n] + \tilde{n}\Bigl[\tilde{n}V[\Pi(t)|N(t)=n] - E[\Pi^2(t)|N(t)=n]\Bigr]%
\end{eqnarray}
Let $Y = \delta \Delta \SIP(t)$.
To obtain a closed-form solution for the binomial gamma process, where the probability of death is $\Pi(t)=1-e^{-\delta \Delta \SIP(t)}$, we need $E[e^{-Y}]$, $V[e^{-Y}]$ and $E[(1-e^{-Y})^2]$, which we can get using the moment generating function $E[e^{zY}] = (\frac{1}{1-z\delta \tau})^{\tau^{-1}\dt}$ for $z \delta \tau < 1$ and $\dt,\lambda,\tau> 0$.
This gives after a Taylor expansion around $\dt=0$
\begin{eqnarray*}
E[e^{-Y}] &=& (1+\delta \tau)^{-\dt/\tau}\\
          &=&  1 - \tau^{-1}\ln(1+\delta\tau)\dt + o(\dt)\\
V[e^{-Y}] &=& E[e^{-2Y}] - E[e^{-Y}]^2 = (1+2\delta \tau)^{-\dt/\tau} - (1+\delta\tau)^{-2\dt/\tau}\\
&=& (1 - \tau^{-1}\ln(1+2\delta\tau)\dt + o(\dt)) - (1 - \tau^{-1}\ln\bigl((1+\delta\tau)^2\bigr)\dt + o(\dt))\\
&=&\tau^{-1}\ln\biggl(\frac{(1+\delta\tau)^2}{1+2\delta\tau}\biggr)\dt + o(\dt))\\
E[(1-e^{-Y})^2] &=& 1 - 2(1 - \tau^{-1}\ln(1+\delta\tau)\dt + o(\dt)) + (1 - \tau^{-1}\ln(1+2\delta\tau)\dt + o(\dt))\\
&=&\tau^{-1}\ln\biggl(\frac{(1+\delta\tau)^2}{1+2\delta\tau}\biggr)\dt + o(\dt)).%
\end{eqnarray*}
Plugging these results in the moment expressions above,
\begin{eqnarray}
\nonumber E[\Delta N(t)|N(t)=n] &=& \tilde{n}\tau^{-1}\ln(1+\delta\tau)\dt + o(\dt)\\
\nonumber V[\Delta N(t)|N(t)=n] &=& \tilde{n}\tau^{-1}\ln{(1+\delta\tau)}\dt +\\
\nonumber                            && +\;\tilde{n}\tau^{-1}\Bigl[\bigl(\tilde{n} - 1\bigr)\ln\biggl(\frac{(1+\delta\tau)^2}{1+2\delta\tau}\biggr)\Bigr]\dt + o(\dt).%
\end{eqnarray}
Since $\frac{(1+\delta \tau)^2}{1+2\delta \tau} > 1$ for $\delta \tau > 0$, it follows
that the process is over-dispersed for $\tilde{n} > 1$ and
equi-dispersed for $\tilde{n} = 1$.
\end{proof}

\begin{proof}[\textbf{Proof of Proposition~\ref{pro:negbin-gamma-proc}} (negative binomial gamma process)] In order to parallel the proof for the binomial gamma process of Proposition~\ref{pro:binom-gamma-proc}, we let the individual birth rate be $\rho$ in this proof, while in the proposition it is represented by $\beta$.
This way we can still use $\beta$ for the gamma scale parameter.
Since $\{N(t)\}$ is a conditional linear birth process, the increment process is negative binomial with parameters number of successes $n$ and success probability $\Pi(t)=e^{-\rho\Delta\SIP(t)}$, i.e
\bean
P(\Delta N(t) = k|N(t)=n,\Delta \SIP(t)) &=& {n + k - 1 \choose k}\Pi(t)^{n}\left(1 - \Pi(t)\right)^k,
\eean
for $k \;\inNat$.
Following a derivation parallel to that of the binomial gamma processes of Proposition~\ref{pro:binom-gamma-proc},
\begin{align*}
&P(\Delta N=k| N(t)=n) = \int\limits_{0}^{\infty}
    {
    {n + k - 1 \choose k}
    \bigl[e^{-x}\bigr]^n
    \bigl[1 - e^{-x}\bigr]^{k}
    \frac{x^{\alpha-1}e^{-x\beta}\beta^{\alpha}}{G(\alpha)}
    }dx\\
&\quad =
    {n + k - 1 \choose k}
    \int\limits_{0}^{\infty}{
    \Biggl[\sum\limits_{j=0}^{k}{
    {k \choose j}
    }(- e^{-x})^{k-j}\Biggr]e^{-xn}
    \frac{x^{\alpha-1}e^{-x\beta}\beta^{\alpha}}{\Gamma(\alpha)}
    }dx\\
&\quad =
    {n + k - 1 \choose k}
    \sum\limits_{j=0}^{k}{
    {k \choose j}
    }(-1)^{k-j}
    \Bigl(1  + \rho \tau(n + k - j)\Bigr)^{-\dt\tau^{-1}}
\end{align*}
The limiting probabilities follow by a Taylor series expansion about $\dt=0$ like in the proof of Proposition~\ref{pro:binom-gamma-proc}.
The moments can be found as follows. Consider the odds against a birth $\Theta(t) = \frac{1-\Pi(t)}{\Pi(t)}$ given the probability of a birth $\Pi(t)$. Then
\begin{eqnarray}
 \nonumber E[\Delta N(t)|N(t)=n]&=& n E[\Theta(t)|N(t)=n] =n E[e^{\rho \Delta \SIP(t)} - 1|N(t)=n] \\
 \nonumber V[\Delta N(t)|N(t)=n]&=& V[n\Theta(t)|N(t)=n] + E[n\Theta(t)(1+\Theta(t))|N(t)=n]\\
 \nonumber                      &=& E[\Delta N(t)|N(t)=n] + n \Bigl[n V[\Theta(t)|N(t)=n] + E[\Theta^2(t)|N(t)=n]\Bigr]%
\end{eqnarray}
Let as for the binomial gamma process $Y = \rho \Delta \SIP(t)$, which follows a gamma distribution with mean $\rho \dt$ and variance $\rho^2 \tau \dt$.
To obtain a closed-form solution for the binomial gamma process, where the odds against a birth is $\Theta(t) = e^{\rho \Delta \SIP(t)} - 1$, we need $E[e^{Y}]$, $V[e^{Y}]$ and $E[(e^{Y}-1)^2]$, which we can get using the moment generating function $E[e^{zY}] = (\frac{1}{1-z\delta \tau})^{\tau^{-1}\dt}$ for $z \delta \tau < 1$ and $\dt,\lambda,\tau> 0$.
This gives after a Taylor expansion around $\dt=0$
\begin{eqnarray*}
E[e^{Y}] &=& (1-\rho \tau)^{-\dt/\tau}\\
          &=&  1 - \tau^{-1}\ln(1-\rho\tau)\dt + o(\dt)\\
V[e^{Y}] &=& E[e^{2Y}] - E[e^{Y}]^2 = (1-2\rho \tau)^{-\dt/\tau} - (1-\rho\tau)^{-2\dt/\tau}\\
&=& (1 - \tau^{-1}\ln(1-2\rho\tau)\dt + o(\dt)) - (1 - \tau^{-1}\ln\bigl((1-\rho\tau)^2\bigr)\dt + o(\dt))\\
&=&\tau^{-1}\ln\biggl(\frac{(1-\rho\tau)^2}{1-2\rho\tau}\biggr)\dt + o(\dt))\\
E[(e^{Y}-1)^2] &=& 1 - 2(1 - \tau^{-1}\ln(1-\rho\tau)\dt + o(\dt)) + (1 - \tau^{-1}\ln(1-2\rho\tau)\dt + o(\dt))\\
&=&\tau^{-1}\ln\biggl(\frac{(1-\rho\tau)^2}{1-2\rho\tau}\biggr)\dt + o(\dt)).%
\end{eqnarray*}
Note that we require that $2\tau\rho < 1$.Plugging this into the moment expressions gives
\begin{eqnarray}
\nonumber E[\Delta N(t)|N(t)n=n] &=& n\tau^{-1}\ln\Bigl(\frac{1}{1-\rho\tau}\Bigr)\dt + o(\dt)\\
\nonumber V[\Delta N(t)|N(t)=n] &=& n\tau^{-1}\ln\Bigl(\frac{1}{1-\rho\tau}\Bigr)\dt +\\
\nonumber                            && +\;n\tau^{-1}\Bigl[\bigl(n - 1\bigr)\ln\biggl(\frac{(1-\rho\tau)^2}{1-2\rho\tau}\biggr)\Bigr]\dt + o(\dt).%
\end{eqnarray}
Since $\frac{(1-\rho \tau)^2}{1-2\rho \tau} > 1$ for $\rho \tau > 0$ and $2\tau\rho < 1$, it follows that the process is also over-dispersed for $n > 1$ and
equi-dispersed for $n = 1$.
\end{proof}

\begin{proof}[\textbf{Proof of Proposition~\ref{pro:binom-beta-proc}} (binomial beta process)] Since $\{N(t)\}$ is the counting process associated with a conditional linear death process, the increment process is binomial with parameters size $\tilde{n}$ and death probability $\Pi(t)$, i.e.
\begin{eqnarray*}
\nonumber P(\Delta N(t){=}k|N(t){=}n,\Delta \Pi(t)) = {\tilde{n}  \choose k }(\Pi(t))^k(1 - \Pi(t))^{\tilde{n} - k},
\end{eqnarray*}
for $k \in \{0,1,\dots,\tilde{n}\}$.

We integrate out the beta noise using the fact that $\Delta N(t)$
conditional on $N(t)=n$ has a beta binomial distribution with the
corresponding parameters.

The beta binomial probability mass function of $\Delta N(t)$ given $N(t)=n$ is
\begin{align}
\nonumber &P(\Delta \CP(t) = k | \CP(t)=n) = \\
\label{eqn:transprob-betabinom1} &=\quad {\tilde{n} \choose k} \frac{\Gamma(\alpha + \beta) \Gamma(k + \alpha) \Gamma(\tilde{n} - k + \beta)}{\Gamma(\alpha) \Gamma(\beta) \Gamma(\alpha + \beta + \tilde{n})}\\
\label{eqn:transprob-betabinom2} &=\quad {\tilde{n} \choose k}\frac{ \Gamma(\alpha + \beta)\Gamma(\alpha)\Gamma(\beta)\Gamma(k)\alpha \{ \frac{\Gamma(\c + \tilde{n} - k)}{\Gamma(\c)}+ O(\dt)\}} {\Gamma(\alpha + \beta)\Gamma(\alpha) \Gamma(\beta) \frac{\Gamma(\c + \tilde{n})}{\Gamma(\c)}}\\
\label{eqn:transprob-betabinom3} &=\quad {\tilde{n} \choose k} \frac{\Gamma(k)\Gamma(\c + \tilde{n} - k)} {\Gamma(\c + \tilde{n})}\c\mu\dt + o(\dt),
\end{align}
for $k \in \{0,\ldots,\tilde{n}\}$.~\eqref{eqn:transprob-betabinom2} follows from~\eqref{eqn:transprob-betabinom1}  via an application of
Lemma~\ref{lemma:gamma-beta-n} in this appendix.
Specifically, using Lemma~\ref{lemma:gamma-beta-n}
with $i = \tilde{n}-k$, it follows that
\bea
\label{eqn:num}\Gamma(\tilde{n} - k + \beta) &=& \Big\{\frac{\Gamma(\c + \tilde{n}-k)}{\Gamma(\c)}+ O(\dt)\Big\}\Gamma(\beta),
\eea
and, since $\alpha + \beta = \c$,
\bea
\label{eqn:den} \Gamma(\alpha + \beta + \tilde{n}) &=& \Gamma(\c + \tilde{n})\\
\nonumber &=& \frac{\Gamma(\c + \tilde{n})}{\Gamma(\c)}\Gamma(\c)\\
\nonumber &=& \frac{\Gamma(\c + \tilde{n})}{\Gamma(\c)}\Gamma(\alpha + \beta).
\eea
Plugging~\eqref{eqn:num} and~\eqref{eqn:den} into~\eqref{eqn:transprob-betabinom1} gives~\eqref{eqn:transprob-betabinom2}. Then, using $\alpha = \c\mu\dt +
o(\dt)$ and canceling terms gives~\eqref{eqn:transprob-betabinom3},
which corresponds to the infinitesimal probabilities.

The moments of a beta binomial distribution are a standard result. Since $\alpha = c(1 - e^{-\delta \dt})$ and $\beta = ce^{-\delta \dt}$ and $\c = \frac{\tilde{n} - 1}{\omega} - 1$ for $\tilde{n} > 1$, Taylor expansions around $\dt=0$ then give
\begin{eqnarray}
\nonumber E[\Delta N(t)|N(t)=n] &=& \tilde{n}\frac{\alpha}{\alpha+\beta}\\
\nonumber                       &=& \tilde{n}\delta\dt + o(\dt)\\
\nonumber V[\Delta N(t)|N(t)=n] &=& \tilde{n}\frac{\alpha \beta}{(\alpha + \beta)^2} \frac{\tilde{n} + \alpha + \beta}{1 + \alpha + \beta}\\
\nonumber                       &=& \tilde{n}(1 - e^{-\delta \dt})e^{-\delta \dt}\frac{\tilde{n} + c}{c+1}\\
\nonumber                       &=& \tilde{n}\delta \dt\;(1+ \omega)+ o(\dt),
\end{eqnarray}
for $\tilde{n} > 1$ and it follows that the binomial beta process is over-dispersed for $\omega > 0$. If $\tilde{n} = 1$ the process is equi-dispersed as
\bean
 V[\Delta N(t)|N(t)=n] &=& \tilde{n}\frac{\alpha \beta}{(\alpha + \beta)^2}\\
                       &=& \tilde{n}\mu\dt + o(\dt)
\eean
\end{proof}

\begin{lemma}\label{lemma:gamma-beta-n}
For $\alpha=\c(1 - e^{-\mu\dt})$, $\beta=\c e^{-\mu\dt}$, $\c > 0$ and $i \in \{1,2,\dots\}$,
\bea
\nonumber \Gamma(\beta + i) &=& \Big\{ \frac{\Gamma(\c + i)}{\Gamma(\c)} + O(\dt)\Big\}\Gamma(\beta).
\eea
\end{lemma}

\begin{proof}
Since $\beta = \c - \alpha$, and by the definition of the gamma
function, for $i \geq 1$,
\bea
\nonumber \Gamma(\beta + i) &=& (\c -  \alpha + (i - 1))\times(\c - \alpha + (i - 2))\times \dots\times(c - \alpha)\times\Gamma(\beta)\\
\nonumber                   &=& \big\{(\c + (i - 1))\times(\c + (i - 2))\times\dots\times (\c) + O(\dt)\big\} \, \Gamma(\beta)\\
\nonumber                   &=& \Big\{ \prod\limits_{j=0}^{i-1}{(\c + j)} + O(\dt)\Big\}\, \Gamma(\beta)\\
\nonumber                   &=& \Big\{ \frac{\Gamma(\c + i)}{\Gamma(\c)} + O(\dt)\Big\}\, \Gamma(\beta).
\eea
\end{proof}


\section{
A lemma required for Theorem~\ref{thm:suf-cond-mult-build-block}}
\label{app:proof-overd-multi}
This technical result is similar to, but slightly different from, standard results on Markov chains.
\begin{lemma}[probability of single jump of any size in multivariate compound processes]\label{lem:prob-Scomp-little-oh}
Let $\{\bm{X}(t)\}$ be a time homogeneous, stable and conservative Markov counting system with associated multivariate counting process $\{\bm{N}(t)\}$ defined by~\eqref{eqn:MCS} and~\eqref{eqn:mass-conservation}, as in Theorem~\ref{thm:suf-cond-mult-build-block}.
Consider a starting time $t$ and let $\Tf$ be the time between $t$ and the first jump time and $\Ts$ be the time between $t + \Tf$ and the second jump time, where jumps can be of size one or more.
Let $S$ be the event of exactly one single jump in any of the $\{N_{ij}(t): i \ne j\}$ counting processes occurring in $[t,t+\dt]$.
Then, letting $\lambda_{\Tf} \equiv \lambda(\bm{x})$ be the rate function of $\{\bm{X}(t)\}$ during $[t,t+\Tf]$ and , and $\Lambda_{\Ts} \equiv \lambda\big(\bm{X}(t + \Tf)\big)$ be this conditional rate function during $[t + \Tf, t + \Ts]$,
\bean
P(S|\bm{X}(t)=\bm{x}, \Lambda_{\Ts}) &=& \lambda_{\Tf} \mbox{e}^{-\dt\Lambda_{\Ts}}\phi(\dt)
\eean
where
\[
\phi(\dt) \equiv \left\{
\begin{array}{l l}
  \dt & \quad \text{if $\lambda_{\Tf}=\Lambda_{\Ts}$}\\
  \frac{1 - \exp\{- \dt (\lambda_{\Tf} - \Lambda_{\Ts})\}}{\lambda_{\Tf} - \Lambda_{\Ts}} & \quad \text{if $\lambda_{\Tf} \neq \Lambda_{\Ts}$}\\
\end{array} \right.
\]
and
\bean
P\big(S|\bm{X}(t)=\bm{x}\big) &=& \dt \sum\limits_{i,j,k}\q[\bm{x},k][ij][] + o(\dt)
\eean
\end{lemma}
%
%
%
\begin{proof}
Start by fixing the random variable $\Lambda_V$ at a given constant, say $\lambda_V$.
Given the Markov property, for the starting time $t$, the densities of the exponential inter-event times are $f_{\Tf}(\tf) = \lambda_{\Tf} \mbox{e}^{-\tf\lambda_\Tf}$ for $\tf>0$
and $f_{\Ts}(\ts) = \lambda_{\Ts} \mbox{e}^{-\ts \lambda_{\Ts}}$ for $\ts>0$.
Then,
\bea
\nonumber P(S|\bm{X}(t)=\bm{x}) &=& P(\Tf < \dt, \Tf+\Ts > \dt) = P(\Tf < \dt, \Ts > \dt - \Tf)\\
\nonumber &=& \int\limits^{\dt}_{0}\int\limits_{\dt-\tf}^{\infty} f_{\Tf,\Ts}(\tf,\ts)d\ts d\tf = \int\limits^{\dt}_{0}\int\limits_{\dt-\tf}^{\infty} \lambda_\Tf \mbox{e}^{-\tf \lambda_{\Tf}} \lambda_\Ts \mbox{e}^{-\ts \lambda_{\Ts}} d\ts d\tf\\
\nonumber &=& \int\limits^{\dt}_{0} \lambda_{\Tf} \mbox{e}^{-\tf \lambda_{\Tf}} \lambda_{\Ts} d\tf \int\limits_{\dt-\tf}^{\infty} \mbox{e}^{-\ts \lambda_{\Ts}}d\ts = \int\limits^{\dt}_{0} \lambda_{\Tf} \mbox{e}^{-\tf \lambda_{\Tf}} \mbox{e}^{-(\dt-\tf)\lambda_{\Ts}}d\tf\\
\label{eqn:prob-one-jump-int} &=& \lambda_\Tf \mbox{e}^{-\dt\lambda_{\Ts}} \int\limits^{\dt}_{0}\mbox{e}^{-\tf (\lambda_{\Tf} -\lambda_{\Ts})}d\tf.%
\eea
If the event rate is not changed by the first event happening (like
it happens in a Poisson process but unlike linear birth or death
processes), then we can write $\lambda_{\Tf} = \lambda_{\Ts}= \lambda$ in~\eqref{eqn:prob-one-jump-int} and
\bea
\label{eqn:phi-h-I} P(S|\bm{X}(t)=\bm{x},\lambda_{\Ts})
          &=& \lambda_{\Tf} \mbox{e}^{-\dt\lambda_{\Ts}}\dt\\
\nonumber &=& \lambda\dt\big(1-\lambda\dt+o(\dt)\big)\\
\label{eqn:prob-one-jump-I} &=& \lambda\dt + o(\dt).
\eea
If  $\lambda_{\Tf} \neq \lambda_{\Ts}$, then from~\eqref{eqn:prob-one-jump-int}
\bea
\label{eqn:phi-h-II} P(S|\bm{X}(t)=\bm{x},\lambda_{\Ts})
          &=& \lambda_{\Tf} \mbox{e}^{-\dt\lambda_{\Ts}}\biggl[\frac{1 - \mbox{e}^{-\dt(\lambda_{\Tf} -\lambda_{\Ts})}}{\lambda_{\Tf} -\lambda_{\Ts}}\biggr]\\
\nonumber &=& \lambda_{\Tf} \frac{\mbox{e}^{-\dt\lambda_{\Ts}} - \mbox{e}^{-\dt\lambda_{\Tf}}}{\lambda_{\Tf} -\lambda_{\Ts}}\\
\nonumber &=& \lambda_{\Tf} \frac{1 - \dt\lambda_{\Ts} + o(\dt) - 1 + \dt\lambda_{\Tf} + o(\dt)}{\lambda_{\Tf} -\lambda_{\Ts}}\\
\nonumber &=& \lambda_{\Tf} \dt\frac{\lambda_{\Tf} - \lambda_{\Ts}}{\lambda_{\Tf} -\lambda_{\Ts}} + o(\dt)\\
\label{eqn:prob-one-jump-II} &=& \lambda_{\Tf} \dt + o(\dt).%
\eea
Combining~\eqref{eqn:phi-h-I} and~\eqref{eqn:phi-h-II}, replacing $\lambda_\Ts$ by $\Lambda_\Ts$ and conditioning on $\Lambda_\Ts$ gives the fist result in the theorem.
For the general case where $\lambda_{\Ts}$ is stochastic, the limit in
\begin{eqnarray*}
\lim\limits_{\dt \downarrow 0}\dt^{-1}P(S|\bm{X}(t)=\bm{x}) = \lim\limits_{\dt \downarrow 0} \dt^{-1}E[P(S|\bm{X}(t)=\bm{x},\Lambda_{\Ts})] &=& E[\lim\limits_{\dt \downarrow 0} \dt^{-1}(\lambda_{\Tf} \dt + o(\dt))] = \lambda_{\Tf},
\end{eqnarray*}
can be passed inside the expected value since $P(S|\bm{X}(t)=\bm{x},\Lambda_{\Ts}) f_{\lambda_{\Ts}} \leq f_{\lambda_{\Ts}}$ and $\int
f_{\lambda_{\Ts}}(r)\;dr=1$.
\end{proof}


\section{Simple MCPs: subordination and multiplicative L\'{e}vy white noise}\label{app:KBDS}
Let $\{M_{\lambda}(t)\}$ be the simple, time homogeneous, conservative and stable MCP with rate function $\lambda:\mathbb{N} \rightarrow R^+$ of Theorem~\ref{thm:KBDS}.
It will be convenient here to write $M(\lambda)(t)$ instead of $M_\lambda(t)$.
Write $\pi^{M(\lambda)}_{n,n+k}(\dt)$ for the integrated increment (or transition) probabilities of  $\{M(\lambda)(t)\}$, defined as
\[
\pi^{M(\lambda)}_{n,n+k}(\dt) \equiv P\Big( \Delta M(t) = k | M(t) = m \Big).
\]
For $\{M(\lambda)(t)\}$, Kolmogorov's Backward Differential System is satisfied \cite{bremaud1999}, i.e.
\begin{eqnarray}
\label{eqn:simple-KBDS}
\frac{d}{d\dt}\pi^{M(\lambda)}_{n,n+k}(\dt) = \Bigl[\; \pi^{M(\lambda)}_{n+1,n+k}(\dt) - \pi^{M(\lambda)}_{n,n+k}(\dt)\;\Bigr] \lambda(m).
\end{eqnarray}
This suggests the following definition of $\{M(\lambda \xi)(t)\}$, a simple MCP $\{M(\lambda)(t)\}$ with multiplicative continuous-time noise in the rate function, where $\{\xi (t)\} \equiv \{dL(t)/dt\}$ for a non-decreasing, L\'{e}vy integrated noise process $\{L(t)\}$ with $L(0)=0$ and $E[L(t)]=t$, as in Theorem~\ref{thm:KBDS}.
Define the process $\{M(\lambda \xi)(t)\}$ by
\[
\pi^{M(\lambda \xi)}_{n,n+k}(\dt) \equiv E\Bigl[\Pi^{M(\lambda \xi)}_{n,n+k}(\dt)\Bigr]
\]
where $\Pi^{M(\lambda \xi)}_{n,n+k}(\dt)$ is specified, by analogy to~\eqref{eqn:simple-KBDS}, as the solution to a stochastic differential equation
\begin{eqnarray}
\label{eqn:stoch-KBDS-sde}
d\Pi^{M(\lambda \xi)}_{n,n+k}(\dt) = \bigl[\; \Pi^{M(\lambda \xi)}_{n+1,n+k}(\dt) - \Pi^{M(\lambda \xi)}_{n,n+k}(\dt) \;\bigr]
\lambda(n)\,dL(h),
\end{eqnarray}
or, essentially equivalently,
\begin{eqnarray}
\label{eqn:stoch-KBDS}
\Pi^{M(\lambda \xi)}_{n,n+k}(\dt) &=& \Pi^{M(\lambda \xi)}_{n,n+k}(0) + \int\limits_{0}^{\dt} \Bigl[\Pi^{M(\lambda \xi)}_{n+1,n+k}(r-) - \Pi^{M(\lambda \xi)}_{n,n+k} (r-) \Bigr] \lambda(m) \;dL(r)
\end{eqnarray}
To give meaning to~\eqref{eqn:stoch-KBDS-sde} and~\eqref{eqn:stoch-KBDS}, it is necessary to define a stochastic integral.
Here, we use the Marcus canonical stochastic integral with Marcus map $\Phi(u,x,y) = \pi^{M(\lambda)}_{n,n+k}(x + uy)$.
The Marcus canonical integral is a stochastic integral developed in the context of L\'{e}vy calculus \cite{applebaum2004}.
It is constructed to satisfy a chain rule of the Newton-Leibniz type (unlike the It\^{o} integral).
In the case of continuous L\'{e}vy processes, the Marcus canonical integral becomes the Stratonovich integral.
For jump processes, the Marcus canonical integral heuristically corresponds to approximating trajectories by increasingly accurate continuous piecewise linear functions.
We interpret~\eqref{eqn:stoch-KBDS-sde} as a stochastic version of~\eqref{eqn:simple-KBDS}.
We then think of $\Pi^{M(\lambda \xi)}_{n,n+k}(\dt)$ as stochastic transition probabilities, conditional on the noise process, giving rise to deterministic transition probabilities $\pi^{M(\lambda \xi)}_{n,n+k}(\dt)$ once this noise is integrated out.

\begin{proof}[\textbf{Proof of Theorem~\ref{thm:KBDS}}(L\'{e}vy white noise and subordination)]
By definition
\[
\pi^{M(\lambda) \circ L}_{n,n+k}(\dt) \equiv E\Bigl[ \pi^{M(\lambda)}_{n,n+k}\big(L(\dt)\big)\Bigr].
\]
Applying Theorem 4.4.28 of Applebaum \cite{applebaum2004} with $f\big(L(\dt)\big) = \Pi^{M(\lambda)}_{n,n+k}\big(L(\dt)\big)$, it follows that $f \in C^3(\R)$ by smoothness of $f$ implied by~\eqref{eqn:simple-KBDS} and then, since $\Pi^{M(\lambda)}_{n,n+k}(0) = 0$, that
\[
\Pi_{n,n+k}^{M(\lambda)}\big(L(\dt)\big) = \int\limits_{0}^{\dt} \Bigl[\Pi^{M(\lambda)}_{n+1,n+k}\big(L(r-)\big) - \Pi^{M(\lambda)}_{n,n+k} \big(L(r-)\big)\Bigr] \lambda(m) \;dL(r),
\]
so that $\Pi^{M(\lambda)}_{n,n+k}\big(L(\dt)\big)$ satisfies~\eqref{eqn:stoch-KBDS}.
Given uniqueness and existence of~\eqref{eqn:stoch-KBDS}, it follows that
\[
\pi^{M(\lambda)}_{n,n+k}\big(L(\dt)\big) \sim \Pi^{M(\lambda \xi)}_{n,n+k}(\dt),
\]
and hence that $\pi^{M(\lambda) \circ L}_{n,n+k}(\dt) = \pi^{M(\lambda \xi)}_{n,n+k}(\dt)$.
\end{proof}




\bibliographystyle{elsarticle-harv}
\bibliography{OCTMCP-bib}







\end{document}